\documentclass[a4paper,reqno,10pt]{amsart}

\usepackage{pifont}
\usepackage{subfiles}
\usepackage{upgreek}
\usepackage{amsfonts}
\usepackage{amsfonts,graphicx,bezier,enumerate,multicol,multirow,subfloat,color,xcolor,amssymb,mathtools,verbatim,appendix,pdfpages,color,colortbl,lscape,longtable,euscript}
\usepackage{amsmath,amsthm,amssymb,colonequals}
\usepackage{alphalph,etoolbox}
\usepackage{indentfirst}
\usepackage{amscd}
\usepackage{setspace}
\usepackage{tikz, tikz-cd}
\usepackage{mathrsfs}
\usepackage{float}
\usepackage{graphicx, subcaption}
\usepackage{wrapfig}
\usepackage{enumitem}
\usepackage{colonequals}
\usepackage{marginnote}
\usepackage{wasysym}
\usepackage[pagewise]{lineno}
\newtheorem{theorem}{Theorem}[section]
\newtheorem{prop}[theorem]{Proposition}
\newtheorem{lemma}[theorem]{Lemma}
\newtheorem{cor}[theorem]{Corollary}
\newtheorem{claim}[theorem]{Claim}

\theoremstyle{definition} 
\newtheorem{definition}[theorem]{Definition}

\newtheorem{example}[theorem]{Example}
\newtheorem{setup}[theorem]{Setup}

\newtheorem{remark}[theorem]{Remark}

\newtheorem{notation}[theorem]{Notation}

\numberwithin{equation}{section}

\newcommand{\bmc}{\begin{multicols}}
	\newcommand{\emc}{\end{multicols}}
\usepackage[colorlinks]{hyperref}
\usepackage{mathrsfs}
\usetikzlibrary{bbox,chains,backgrounds,arrows,decorations.pathmorphing,decorations.pathreplacing,shapes.geometric,shapes.misc,decorations.markings,decorations.fractals,calc,patterns,matrix,positioning,spy}

\tikzset{
	cvertex/.style={circle,draw=blue,inner sep=1pt,outer sep=3pt},
	cvertex1/.style={circle,draw=green,inner sep=1pt,outer sep=3pt},
	vertex/.style={circle,fill=black,inner sep=1pt,outer sep=3pt},
	DBs/.style={circle,draw=black,circle,fill=black,inner sep=0pt, minimum size=3pt},
	DB/.style={circle,draw=black,circle,fill=black,inner sep=0pt, minimum size=4pt},
	DWs/.style={circle,draw=black,circle,fill=white,inner sep=0pt, minimum size=3pt},
	DWds/.style={circle,draw=black,densely dotted,circle,fill=white,inner sep=0pt, minimum size=3pt},
	DWds1/.style={circle,draw=white,fill=pink,inner sep=0pt, minimum size=3pt},
	DWds2/.style={circle,draw=white,fill=purple,inner sep=0pt, minimum size=3pt},
	DWds3/.style={circle,draw=white,fill=blue,inner sep=0pt, minimum size=3pt},
	DWds4/.style={circle,draw=white,fill=red,inner sep=0pt, minimum size=3pt},
	DWds5/.style={circle,draw=white,fill=green,inner sep=0pt, minimum size=3pt},
	DW/.style={circle,draw=black,inner sep=0pt, minimum size=4pt},
	tvertex/.style={inner sep=1pt,font=\scriptsize},
	gap/.style={inner sep=0.5pt,fill=white},
	Bgap/.style={circle,draw=black,circle,inner sep=-1pt,fill=black},
	mid/.style={inner sep=0.5pt},
	Ggap/.style={inner sep=0.5pt,fill=green!40!black!20}}

\tikzset{
	W/.style={circle,draw=black!60!white,circle,fill=white,inner sep=0pt, minimum size=4pt},
	B/.style={circle,draw=black!60!white,circle,fill=black!60!white,inner sep=0pt, minimum size=4pt},
}
\tikzset{|/.tip={Bar[width=.8ex,round]}}

\newcommand{\Anohat}{\operatorname{\mathrm{A}}}
\newcommand{\Bnohat}{\operatorname{\mathrm{B}}}
\newcommand{\Aphat}{\operatorname{\vphantom{\hat{\mathrm{A}}}\mathrm{A}}}
\newcommand{\Bphat}{\operatorname{\vphantom{\hat{\mathrm{B}}}\mathrm{B}}}
\newcommand{\Lambdaphat}{\operatorname{\vphantom{\hat{\Lambda}}\Lambda}}

\newcommand\m{\mathfrak{m}}
\newcommand\p{\mathfrak{p}}

\newcommand\Db{\mathop{\rm{D}^b}}
\newcommand\Dbfl{\mathop{\rm{D}^{\mathrm{b}}_{\fl}}}
\newcommand{\D}{\operatorname{D}}

\newcommand{\smallmod}{\operatorname{mod}}
\newcommand{\Mod}{\operatorname{Mod}}
\newcommand{\RA}{\operatorname{RA}}
\newcommand{\LA}{\operatorname{LA}}
\newcommand{\TW}{\operatorname{\mathsf{Twist}}}
\newcommand{\Geo}{\operatorname{\mathsf{GeoTwist}}}
\newcommand{\Id}{\operatorname{\mathsf{Id}}}
\newcommand{\Ext}{\operatorname{Ext}}
\newcommand{\Supp}{\operatorname{Supp}}
\newcommand{\redu }{\operatorname{red}}
\newcommand{\End}{\operatorname{End}}
\newcommand{\Max}{\operatorname{Max}}

\newcommand{\smallproj}{\operatorname{proj}}
\newcommand{\Hom}{\operatorname{Hom}}

\newcommand{\coh}{\operatorname{coh}}
\newcommand{\Qcoh}{\operatorname{Qcoh}}
\newcommand{\con}{\operatorname{\mathrm{con}}}
\newcommand{\RHom}{\operatorname{{\mathbf R}Hom}}
\newcommand{\res}{\operatorname{res}}
\newcommand{\refl}{\operatorname{ref}}

\newcommand{\fl}{\operatorname{\mathsf{fl}}}
\def\Db{\mathop{\rm{D}^b}\nolimits}
\def\Kb{\mathop{\rm{K}^b}\nolimits}

\renewcommand{\*}{%
	\discretionary {\thinspace\the\textfont2\char1}{}{}%
}

\newcommand{\aff}{\mathsf{aff}}

\newcommand{\scrA}{\EuScript{A}}

\newcommand{\scrC}{\EuScript{C}}

\newcommand{\scrE}{\EuScript{E}}

\newcommand{\scrH}{\EuScript{H}}

\newcommand{\scrJ}{\EuScript{J}}

\newcommand{\scrL}{\EuScript{L}}

\newcommand{\scrO}{\EuScript{O}}
\newcommand{\scrP}{\EuScript{P}}

\newcommand{\scrR}{\EuScript{R}}

\newcommand{\scrT}{\EuScript{T}}
\newcommand{\scrU}{\EuScript{U}}
\newcommand{\scrV}{\EuScript{V}}
\newcommand{\scrW}{\EuScript{W}}
\newcommand{\scrX}{\EuScript{X}}
\newcommand{\scrY}{\EuScript{Y}}

\DeclareMathOperator{\Spec}{Spec}

\definecolor{Pink}{RGB}{230 56 243}
\definecolor{Blue}{RGB}{0 19 147}
\definecolor{Green}{RGB}{66 147 41}
\definecolor{Grey}{RGB}{102 102 102}
\definecolor{Orange}{RGB}{237 107 45}
\definecolor{Red}{RGB}{234 53 159}
\def\mythick{0.5}
\def\EeightThreeScalex{0.43}
\def\EeightThreeScaley{0.5}

\newcommand{\EeightThreeNoLattice}[1][]{%
	\begin{tikzpicture}[xscale=\EeightThreeScalex,yscale=\EeightThreeScaley]
		\coordinate (A) at (4,0);
		\coordinate (B) at (4,2);
		\coordinate (C) at (4,4);
		\coordinate (Op) at (8,-2);
		\coordinate (Ap) at (8,0);
		\coordinate (Bp) at (8,2);
		\coordinate (m1) at (6,4);
		\coordinate (m2) at (6,-2);
		\coordinate (u1) at (intersection cs:
		first line={(C) -- (Bp)},
		second line={(m1) -- (m2)});
		\coordinate (u2) at (intersection cs:
		first line={(A) -- (Op)},
		second line={(m1) -- (m2)});    
		\draw (A)--(Op);
		\draw (Op)--(Bp);
		\draw (Bp)--(C);
		\draw (C) -- (A);
		\draw (B)--(Bp);
		\draw (A)--(Ap);
		\draw (C)--(Op);
		\draw (A)--(Bp);
		\draw (A) -- (u1);
		\draw (Bp) -- (u2);
		\draw (B) -- (u1);
		\draw (Ap) -- (u2);
		\draw (Ap) -- (B);
		\draw (u1) -- (u2);
	\end{tikzpicture}
}

\begin{document}
	
	\title{Group Actions from Algebraic Flops}
	\author{Caroline Namanya}
	\address{Caroline Namanya, Makerere University, Kampala, Uganda \& University of Glasgow,  Glasgow, Scotland 
	}
	\email{caronamanya97@gmail.com}
	\begin{abstract}
		This paper constructs derived  autoequivalences associated to  an algebraic flopping contraction \(X\to X_{\con}, \) where \(X\) is quasi-projective with only mild singularities. These functors are  constructed naturally  using  bimodule cones, and we prove these cones are locally two-sided tilting complexes by using  the local-global properties  and a key commutative diagram. The main result is that these autoequivalences combine to give an action of the fundamental  group of an associated infinite hyperplane arrangement on the derived category of \(X.\) This  generalises and simplifies \cite{DW3}, by finally  removing reliance on  subgroups, and it also lifts many other results from the complete local setting.
	\end{abstract}
	\maketitle
	\parindent 20pt
	\parskip 0pt

	\section{Introduction}
	\subsection{Motivation}Over the years, studying autoequivalences of derived categories  of coherent sheaves  has been of increasing interest in algebraic geometry. The study of these autoequivalences has been on  the  rise,  with a particular interest in the derived categories and invariants encoded in  birational transformations  of  3-fold flops.

	In the pursuit to understand  derived autoequivalences of 3-fold flops, Donovan and Wemyss	in \cite{DW1}   introduced  the \textit{contraction algebra}  which is a noncommutative deformation algebra  associated to any contractible rational curve in any 3-fold. This  new invariant is used to  describe  a particular  derived autoequivalence  associated to  a general flopping curve, thus creating a link between the birational geometry of 3-fold flops and noncommutative algebras.
	
	By \cite{DW3}, after passing to the complete local setting,  to a global flopping contraction \(X \to X_{\con}\), where \(X\)  has only mild singularities (see Setup \ref{global set up}), one  can associate   two combinatorial objects: a finite hyperplane arrangement $\scrH$ and an infinite hyperplane arrangement $\scrH^{\aff}$  (see \S \ref{specific hyperplanes}). Furthermore, in \cite{DW3}, these arrangements that arise locally are shown to control the geometry via the action of   a subgroup $K$ of the fundamental group of the complexified complement  \( \pi_{1}(\mathbb{C}^{n} \backslash \scrH_\mathbb{C})\)   on the derived category   \(\Db(\coh X).\) However, \cite{N} proves the key  limitation that  in general \(K\neq \pi_{1} \),  giving  rise to the question whether it is possible to realise an action of the full fundamental group on the   derived category in the global setup.
	
	The main point of this paper  is that the subgroup \(K\) is not important, and indeed we prove that it is possible to construct a full action of  \(\pi_{1} \) directly. A hint  of how to do this  appears in \cite{A}, where  bimodule cone autoequivalences are  constructed using contraction algebra technology. It turns out that this vastly generalises. 
	\subsection{Main results}Before stating our main results, we remark that  all the hard work in this paper goes into establishing the Zariski local situation. Here,  we consider   an algebraic flopping contraction \(U\to \Spec R,\) where \(U\) has only  mild singularities (see Setup \ref{01 of 23 april 2023}). It is well known, even in this algebraic setting by \cite{VdB},  that \(U\) is derived equivalent to some \(R\)-algebra $\Lambda.$  By exploiting this derived equivalence we  construct  derived autoequivalences on \(U\) from  bimodule cones on $\Lambda$,  as explained below. 
	
	The key point is   that  in this situation of mild singularities, every monodromy around a wall of a chamber in the associated to  \(\scrH\) or \(\scrH^{\aff}\) gives rise to a derived autoequivalence. The following is our first result. 
	\begin{theorem}[=\ref{Twist is an equiv}, \ref{action on U}] \label{submain thm in intro}
		Consider an algebraic flopping contraction \(U \to \Spec R\) as in Setup \textnormal{\ref{01 of 23 april 2023}}, with associated  finite hyperplane arrangement \(\scrH\) and infinite hyperplane arrangement \(\scrH^{\aff}.\)
		\begin{enumerate}
			\item For every monodromy around a wall of a chamber in \(\scrH\) or \(\scrH^{\aff},\) there exists   an autoequivalence of \(\Db(\coh U).\)
			\item There exists  a commutative diagram of  group homomorphisms \[\begin{array}{c}
				\begin{tikzpicture}
					\node (a1) at (0,0) {$  \uppi_{1}(\mathbb{C}^{n} \setminus \scrH_{\mathbb{C}})  $};
					\node (a2) at (0,-1.5) {$  \uppi_{1}(\mathbb{C}^{n} \setminus \scrH^{\aff}_{\mathbb{C}}) $};
					\node (b1) at (3,0) {$ \mathrm{Auteq}\Db(\coh U)$};
					\draw[->] (a1) to node[below] {$ $} (a2);		
					\draw[->] (a2) to node[below] {$ \scriptstyle g^{\aff}$} (b1);		
					\draw[->] (a1) to node[above] {$ \scriptstyle g$} (b1);				
				\end{tikzpicture}
			\end{array}\] 
		\end{enumerate}
	\end{theorem}
	The  autoequivalences in  Theorem \ref{submain thm in intro}   rely on constructing two-sided tilting complexes on \(\Lambda\)  via bimodule cones.  The guiding claim is that  there is   a map from the  group \(\uppi_{1}\) which is generated by monodromies  around hyperplanes   in the associated \(\scrH\) or \(\scrH^{\aff}\),  to \(\mathrm{Auteq}\Db(\coh U)\),  which is a well defined group homomorphism. 
	
	It should be noted that,	in addition to the  subgroup \(K\)  defined in \cite{DW3} being problematic, the generalisation  to the algebraic setting  comes with two costs, namely  there is lack of Krull-Schmidt (since \(R\) is not complete local), and  there are  no obvious algebraic objects which correspond to the chambers of \(\scrH\) and \(\scrH^{\aff},\) and so  no obvious  way of producing autoequivalences via monodromy around a wall of a chamber.
	
	With   Theorem \ref{submain thm in intro}  in hand, we then globalise to 3-fold flopping contractions  \(h\colon X \to  X_{\con} \)  as in Setup \ref{global set up}. As explained in $\S$\ref{section 8}, in this global setting there is a product of finite arrangements \(\mathfrak{H}\) and a product of infinite arrangements \(\mathfrak{H}^{\aff}.\)   The following is our main result. \begin{theorem}[=\ref{global group action}]\label{global theorem in intro} Under the global assumptions of Setup   \textnormal{\ref{global set up}}, there exists a commutative diagram of group homomorphisms
		\[\begin{array}{c}
			\begin{tikzpicture}
				\node (a1) at (0,0) {$  \uppi_{1}(\oplus \mathbb{C}^{ n_{k}} \setminus \mathfrak{H}_{\mathbb{C}})  $};
				\node (a2) at (0,-1.5) {$  \uppi_{1}(\oplus \mathbb{C}^{ n_{k}} \setminus \mathfrak{H}^{\aff}_{\mathbb{C}}) $};
				\node (b1) at (4,0) {$ \mathrm{Auteq}\Db(\coh X)$};
				\draw[->] (a1) to node[below] {$ $} (a2);		
				\draw[->] (a2) to node[below] {$\scriptstyle m^{\aff}$} (b1);		
				\draw[->] (a1) to node[above] {$\scriptstyle  m$} (b1);				
			\end{tikzpicture}
		\end{array}\] 
	\end{theorem}
	We remark that  Theorem \ref{global theorem in intro}  generalises \cite[Theorem 1.2]{DW3} in two ways: in the finite case \(\mathfrak{H}\) it removes the assumption that the curves are individually floppable, and the infinite case \(\mathfrak{H}^{\aff}\) is new.
	\subsection{Bimodule  Constructions}\label{intro to bimodule constructions} The autoequivalences in Theorems  \ref{submain thm in intro} and  \ref{global theorem in intro} are constructed via bimodule  cone techniques. 	For \(f\colon U\to \Spec R\)   an  algebraic flopping contraction of Setup \ref{01 of 23 april 2023}, we know that \(U\) is derived equivalent to some \(R\)-algebra \(\Lambda.\)  It is therefore sufficient  to build  a derived autoequivalence on \(\Lambda\) and then track these  back  to  get derived autoequivalences on \(U.\)

	Before describing bimodule cone constructions, we first describe  the notion of monodromy around a wall \(i\) through an atom \(\upalpha\) in the associated \(\scrH\) or \(\scrH^{\aff}\), since this will correspond to  a derived  autoequivalences on \(\Lambda\). After reverting to the formal fibre in the sense of  $\S$\ref{specific hyperplanes}, there is an associated hyperplane arrangement which depends on Dynkin data (see $\S$\ref{specific hyperplanes} and \cite{IW7}). Given a wall \(i\) in some chamber \(D\) we can find a shortest path \(\upalpha\) (called an \emph{atom}) from a fixed chamber \(C_{+}\) to \(D\), followed by monodromy  around the wall \(i\), followed by the (inverse of) the shortest path \(D\) to \(C_{+}\). This data is equivalent to specifying the pair \((\upalpha, i),\) and one example is  illustrated in  the following diagram.
	
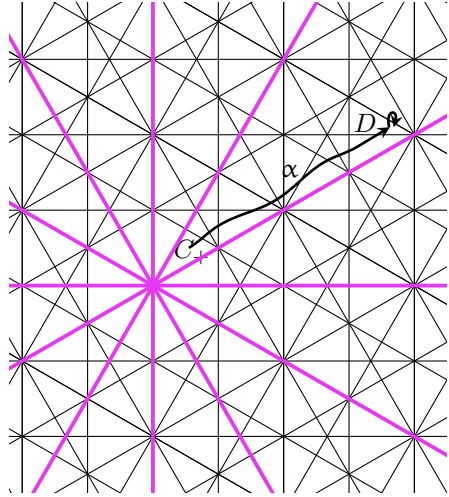
\[
	\begin{tikzpicture}
		\clip (-2.75,-1.75) rectangle (3,2.25);
		\node at (0,0)
		{
			$\begin{tikzpicture}[>=stealth,xscale=\EeightThreeScalex,yscale=\EeightThreeScaley, extended line/.style={shorten >=-#1,shorten <=-#1}]
				\foreach \y in {-4,-2,-1,0,1,2,4}
				\foreach \x in {-4,-2,0,2,4}
				{
					\node at (4*\x,4*\y) {$\EeightThreeNoLattice$};
				}
				\foreach \y in {-3,-2,-1,0,1}
				\foreach \x in {-3,-1,1,3}
				{
					\node at (4*\x,2+4*\y) {$\EeightThreeNoLattice$};
				}
				\node (b1) at (-0.8,-0.1) {$ C_{+}$};
				\node (b2) at (4.5,3.3) {$D$};
				\node (b3) at (5.3,2.95) {}; 
				\draw[line width=\mythick mm, Pink,extended line=5cm] (0,0) -- (-4,-2);
				\draw[line width=\mythick mm, Pink,extended line=5cm] (0,-1) -- (-4,-1);
				\draw[line width=\mythick mm, Pink,extended line=5cm] (0,-2) -- (-4,0);
				\draw[line width=\mythick mm, Pink,extended line=5cm] (-1,-2.5) -- (-3,0.5);
				\draw[line width=\mythick mm, Pink,extended line=5cm] (-3,-2.5) -- (-1,0.5);
				\draw[line width=\mythick mm, Pink,extended line=5cm] (-2,-3) -- (-2,1);
				\draw[line width = 1 pt, in=,out=60,loop above,segment length=6mm,->] (b3) to (b3);
				\draw [line width = 1 pt,->,decorate, decoration={snake,amplitude=0.5mm,segment length=16mm,post length=1mm}] (-0.9,0) -- node [above] {$\upalpha$}(5.25,3.2); 
			\end{tikzpicture}$};
	\end{tikzpicture}
	\]
	We then associate an endomorphism  algebra \(\Bnohat\) to the chamber \(D.\)  This algebra inherits  \(n+1\) primitive  idempotents \(e_{i},\)  (for more details see  Setup \ref{bi setup}), and there is a ring homomorphism \(\Bnohat \to \Bnohat_{i}, \)  where \(\Bnohat_{i}=\Bnohat/ (1-e_{i}).\)
	
	For \(f\colon U\to \Spec R\)   an  algebraic flopping contraction of Setup \ref{01 of 23 april 2023}, to construct an autoequivalence on \(U\) associated to the above monodromy, we consider an endomorphism  algebra \(\Lambda\)  that is derived equivalent to \(U,\) as in \S \ref{geometric setup}.  For any choice of \((\upalpha,i)\) as above  (see also Setup \ref{bi setup}), we first define a functor diagram 
	\[\begin{array}{c}
		\begin{tikzpicture}[looseness=1,bend angle=20]
			\node (a1) at (0,0) {$\D(\Bnohat_{i})$};
			\node (a2) at (3,0) {$\D( \Lambda)$};
			\draw[->] (a1) -- node[gap] {$\scriptstyle \Uppsi_{\upalpha,i}$} (a2);
			\draw[->,bend right] (a2) to node[above] {$\scriptstyle \Uppsi_{\upalpha,i}^{\LA}$} (a1);
			\draw[->,bend left] (a2) to node[below] {$\scriptstyle \Uppsi_{\upalpha,i}^{\RA}$} (a1);										
		\end{tikzpicture}
	\end{array},\] and then construct  a natural transformation \(\Uppsi_{\upalpha,i}\circ \Uppsi_{\upalpha,i}^{\RA}\to \Id_{\Lambda}\) by exhibiting  a certain bimodule map
	\[
	\Lambda \to {}_{\Lambda}\hat{\Lambda} \otimes^{\bf{L}}_{\hat{\Lambda}}  \left(   \mathbb{M}\otimes^{\bf{L}}_{\Aphat}  \uptau_{\upalpha}\otimes^{\bf{L}}_{\Bphat} Z_{\upalpha,i}\right)\otimes^{\bf{L}}_{\hat{\Lambda}} \hat{\Lambda}_{\Lambda},
	\]	in the derived category of \(\Lambda\)-\(\Lambda\) bimodules. Thus taking the cone in the  derived category of \(\Lambda\)-\(\Lambda\) bimodules  gives a triangle 
	\[
	C_{\upalpha,i} \to {}_{\Lambda}\Lambda_{\Lambda} \to {}_{\Lambda}\hat{\Lambda} \otimes^{\bf{L}}_{\hat{\Lambda}} \left(  \mathbb{M}\otimes^{\bf{L}}_{\Aphat}  \uptau_{\upalpha}\otimes^{\bf{L}}_{\Bphat} Z_{\upalpha,i} \right)\otimes^{\bf{L}}_{\hat{\Lambda}} \hat{\Lambda}_{\Lambda}\to  C_{\upalpha,i}[1].
	\] Define \(\TW_{\upalpha,i}\colonequals   \RHom ({}_{\Lambda}C_{\upalpha,i}{}_{\Lambdaphat},-) \) and   \(\TW^{*}_{\upalpha,i}\colonequals  -\otimes^{\bf{L}}_{\Lambda} C_{\upalpha,i}{}_{\Lambdaphat}.\)   The following  is the main technical  result.
	\begin{theorem}[=\ref{02 0f 06 march 2023}]\label{local and zarisiki equivalence}
		\({}_{\Lambda}C_{\Lambda}={}_{\Lambda}C_{\upalpha,i}{}_{\Lambdaphat}\)    is a two-sided tilting complex,  giving  rise to the autoequivalence \(\TW_{\upalpha,i}\)  of \(\Db (\smallmod  \Lambda).\)  Furthermore, this  fits in  a commutative diagram \[
		\begin{array}{c}
			\begin{tikzpicture}
				\node (a1) at (0,0) {$\D(\hat{\Lambda})$};
				\node (a2) at (2,0) {$\D(\Anohat) $};
				\node (a3) at (4,0) {$\D(\Bnohat)$};
				\node (a4) at (6,0) {$\D(\Bnohat)$};
				\node (a5) at (8,0) {$\D(\Anohat)$};
				\node (a6) at (10,0) {$\D(\hat{\Lambda})$};
				\node (b1) at (0,-1.5) {$\D(\Lambda)$};
				\node (b2) at (10,-1.5) {$\D(\Lambda)$};
				\draw[->] (a1) to node[above] {$\scriptstyle  \mathbb{F}^{-1}$} (a2);
				\draw[->] (a2) to node[above] {$\scriptstyle \Upphi_{\upalpha}^{-1}$} (a3);
				\draw[->] (a3) to node[above] {$\scriptstyle {\Upphi}_{i}^{2}$} (a4);
				\draw[->] (a4) to node[above] {$\scriptstyle \Upphi_{\upalpha}$} (a5);
				\draw[->] (a5) to node[above] {$\scriptstyle \mathbb{F}$} (a6);
				\draw[->] (b1) to node[above] {$\scriptstyle\TW_{\upalpha,i}$} (b2);
				\draw[<-] (b1) to node[left] {$\scriptstyle \mathrm{F}$} (a1);
				\draw[<-] (b2) to node[right] {$\scriptstyle  \mathrm{F} $} (a6);
			\end{tikzpicture}			
		\end{array}
		\]  where the top functor is a composition of mutation and Morita equivalences, and \(F\) is  restriction of scalars.
	\end{theorem}

	The commutative diagram above intertwines the complete local monodromy  (in the top functor) with the Zariski local autoequivalence \(\TW_{\upalpha,i}.\) In case of the finite arrangement \(\scrH,\) these twist functors also behave well with respect to the contraction algebra equivalences of August \cite{A}. As notation, let \(\Lambda_{\con}\) be the contraction algebra of \cite{DW1}.
	
	\begin{cor}[=\ref{jenny and caro diag }]\label{jenny and caro diag in intro}
		For any choice of \((\upalpha,i)\) in the finite arrangement \(\scrH\) the following diagram commutes
		\[
		\begin{tikzpicture}
			\node (a1) at (0,0) {$\D(\hat{\Lambda}_{\mathrm{con}})$};
			\node (a2) at (3,0) {$\D(\hat{\Lambda}_{\mathrm{con}})$};
			\node (c1) at (0,-1.5) {$\D(\Lambda)$};
			\node (c2) at (3,-1.5) {$\D(\Lambda)$};
			\draw[->] (a1) to node[above] {$\scriptstyle  \mathrm{J}_{\upalpha,i}$} (a2);
			\draw[->] (c1) to node[above] {$\scriptstyle \TW_{\upalpha,i}$} (c2);
			\draw[->] (a1) to node[right] {$ $} (c1);
			\draw[->] (a2) to node[left] {$ $} (c2);							
		\end{tikzpicture}
		\]where \(\mathrm{J}_{\upalpha,i}\) are the compositions of the standard equivalences of \cite{A}, recalled  in $\S$\textnormal{\ref{section Key Commutative Diagram}}.
	\end{cor}

	The key point of Theorem \ref{local and zarisiki equivalence} is that it is much more general.  Contraction algebras only exist in the finite \(\scrH,\) whereas  the autoequivalences in the infinite \(\scrH^{\aff}\) in Theorem \ref{local and zarisiki equivalence} are much more general.
	
	\subsection*{Conventions} In this paper, the base field is $ \mathbb{C}$.  Throughout \(R\) is a normal  isolated cDV singularity. We will drop both super and subscripts from the tensors as much as possible, whilst still maintaining their natural meanings and contexts in the background.  When tensoring by bimodules, we will suppress the obvious module structure, so for a bimodule \( {}_{\Lambda}X_{\Gamma}, \) we write \( \otimes_{\Lambda} X_{\Gamma}.\)
	\subsection*{Acknowledgements}This work forms part of the author's  PhD,  and was funded  by  an IMU  Breakout Graduate Fellowship. The author recognizes support from  the ERC Consolidator Grant 101001227 (MMiMMa)  for  a one year visit as a PhD student 2021–22 at University of Glasgow, where part of this work was done. The author is immensely grateful to her PhD supervisors Michael Wemyss and David Ssevviiri for their helpful guidance, and thanks Wahei Hara and Jenny August for the many helpful discussions, Will Donovan for comments on a previous draft, and the referee for helpful comments.
	
	\section{Preliminaries}
	The following section gives definitions, terminologies and notation that will  be used.
	\subsection{Geometric setup}\label{geometric setup}
	Recall that a projective birational morphism \(f\colon  X \to Y\) is small if the exceptional locus    has  codimension   at least two. When the dimension is three, this translates into  \(f \) does not contract a divisor.
	
	\begin{definition}
		A \textit{flop} is a commutative diagram \[\begin{tikzcd}[sep=normal,cramped]
			X^{-} \arrow[dr,  "f^{-}"'] \arrow[rr, dashed,"g"]{}
			& & X^{+} \arrow[dl, "f^{+}"] \\
			& Y \end{tikzcd}\]
		where \(f^{\pm}\) are small projective  birational morphisms, and the canonical bundles \(\omega_{X^{\pm}}\) are trivial over \(Y\).
	\end{definition}
	We will refer to   \(f^{-}\) and \(f^{+}\) as flopping contractions.  For 3-folds,  a flop is  a process  of cutting out rational curves \(C_{i}\) and replacing them with a union of other rational curves,
	without contracting any divisors.
	\begin{setup}\cite[Setup 2.3]{DW3}\label{01 of 23 april 2023}
		Suppose that \(f \colon U \to  \Spec R \) is a flopping contraction  which is an isomorphism away from \emph{precisely one point} \(\mathfrak{m} \in \Max R\). We assume that \(U\) has only Gorenstein terminal singularities. As notation, above \(\mathfrak{m}\) is a connected chain \(C\) of \(n\) curves with reduced scheme structure \(C^{\redu}=\bigcup_{j=1}^{n}C_{j}\) such that each \(C_{j}=\mathbb{P}^{1}\)
	\end{setup}

	Under Setup \ref{01 of 23 april 2023}, by  \cite[Theorem A]{VdB} there is a tilting bundle \(\mathcal{V}=\mathcal{O}_{U} \oplus \mathcal{N}\) on \(U\) inducing a derived equivalence \[\Db(\coh U) \xrightarrow{\RHom(\mathcal{V},-)} \Db(\smallmod \Lambda ),\] where  \(\Lambda\colonequals \End_{U}(\mathcal{V}) \cong \End_{R}(f_{*}\mathcal{V})=\End_{R}(R\oplus f_{*}\mathcal{N}), \)  by  \cite[4.2.1]{VdB}.
	
	\subsection{ General hyperplane arrangements}\label{arra section }
	A real hyperplane arrangement,  written $\mathcal{H},$ is a finite set of hyperplanes in \(\mathbb{R}^{n}\).  Such an arrangement is called Coxeter if it arises as the set of reflection hyperplanes of a finite real reflection group. $\mathcal{H}$  is simplicial if the intersection of all the hyperplanes in \(\mathcal{H}\) is \({0}\) and all chambers  \(\mathbb{R}^{n} \backslash \mathcal{H} \) of the arrangement are open simplicial cones.  All Coxeter arrangements are simplicial, but the converse is  false. 
	\begin{definition}
		\cite[Definition 2.6] {AW} 
		Let $\Gamma_{\mathcal{H}}$	 be the  \emph{oriented graph} associated to the hyperplane arrangement $\mathcal{H}$ defined as follows. The vertices of $\Gamma_{\mathcal{H}}$ are the chambers of \(\mathcal{H},\) i.e., the connected components of $\mathbb{R}^{n}\setminus \mathcal{H}$. There is a unique arrow $a \colon v_{1}\rightarrow v_{2}$ from chamber $v_{1}$ to chamber $v_{2}$ if the chambers
		are adjacent, otherwise there is no arrow.
	\end{definition}

	By definition, if there is an arrow $a \colon v_{1}\rightarrow v_{2}$, then there is a unique arrow $b \colon v_{2} \rightarrow v_{1}$ with the opposite direction of $a$. For an arrow $a \colon v_{1}\rightarrow v_{2}$, call \(s(a) \colonequals v_{1}\) the source of \(a,\) and  \( t(a) \colonequals v_{2}\) the target of \(a\).
	A positive path of length $n$ in $\Gamma_{\mathcal{H}}$ is a formal symbol
	$$p = a_{n} \circ \ldots \circ a_{2} \circ a_{1}$$	whenever there exists a sequence of vertices $v_{0},\ldots, v_{n}$ of $\Gamma_{\mathcal{H}}$ and arrows $a_{i} \colon v_{i-1} \rightarrow v_{i}$ in
	$\Gamma_{\mathcal{H}}$. Set $s(p) \colonequals v_{0}, t(p)\colonequals v_{n}$, and call \(l(p)\colonequals n\) the length of \(p.\) If $q = b_{m} \circ \ldots \circ b_{2}\circ b_{1}$
	is another positive path with $t(p) = s(q)$, we consider the formal symbol
	$$q \circ p\colonequals b_{m}\circ  \ldots \circ b_{2} \circ b_{1} \circ a_{n}\circ \ldots \circ a_{2} \circ a_{1}$$
	and call it the composition of $p$ and $q$.
	\begin{definition}\label{atom definition}
		\cite[Definition 2.6]{HW} A positive path is called\textit{ minimal} if there is no positive path in $\Gamma_{\mathcal{H}}$ of smaller length, and with the same endpoints. The positive minimal paths are called \textit{atoms}. A positive path is called \textit{reduced} if it does not cross any hyperplane twice.
	\end{definition}
	
	\subsection{Specific hyperplane arrangements}\label{specific hyperplanes}
	Returning to the setting of  \ref{01 of 23 april 2023},  set \(\scrR= \hat{R}\)  to be the completion of \(R\) at the unique singular point \(\mathfrak{m}.\) The natural morphism \(R\to \hat{R}=\scrR,\) induces the following diagram
	\begin{eqnarray}\label{formal fibre diagram}
		\begin{tikzpicture}[yscale=1.25]
			\node (A) at (-1,0) {$\scrU$}; 
			\node (B) at (1,0) {$U$};
			\node (C) at (-1,-1) {$\Spec \scrR$}; 
			\node (D) at (1,-1) {$\Spec R$};
			\draw[->,dashed] (A) to node[above] {$\scriptstyle $} (B);
			\draw[->] (C) to node[above] {$\scriptstyle $} (D);
			\draw[->,dashed] (A) --  node[left] {$\scriptstyle \upvarphi$}  (C);
			\draw[->] (B) --  node[mid right] {$ \scriptstyle \ref{01 of 23 april 2023}$} node[mid left] {$\scriptstyle f$} (D);
		\end{tikzpicture}
	\end{eqnarray}
	where \(\upvarphi\) is  called the formal fibre. The morphism \(\scrU \to \Spec \scrR \)  is a formal flopping contraction, and   for a generic element  \( g \in \scrR, \) slicing   induces the following diagram
	\[
	\begin{tikzpicture}[yscale=1.25]
		\node (A1) at (-2,0.5) {$\scrX$}; 
		\node (A) at (-1,0) {$\scrY$}; 
		\node (B) at (1,0) {$\scrU$};
		\node (C) at (-1,-1) {$\Spec (\scrR/g)$}; 
		\node (D) at (1,-1) {$\Spec \scrR$};
		\draw[->, dashed] (A) to node[above] {$\scriptstyle $} (B);
		\draw[->] (C) to node[above] {$\scriptstyle $} (D);
		\draw[->] (A1) to node[above] {$\scriptstyle $} (A);
		\draw[->] (A1) to node[above] {$\scriptstyle $} (C);
		\draw[->,dashed] (A) --  node[right] {$\scriptstyle \uppsi$}  (C);
		\draw[->] (B) --  node[right] {$\scriptstyle \upvarphi$}  (D);
	\end{tikzpicture}
	\]
	By Reid's general elephant \cite{R1}, \(\scrR/g\) is an $\mathrm{ADE}$ Kleinian singularity, \(\uppsi\) is a partial resolution morphism, and  \(\scrX\) is the minimal resolution. By the McKay correspondence, the exceptional curves \(C_{i}\)  of \(\scrX\) are indexed by vertices  \(i\) in a  Dynkin diagram \(\Updelta.\) We will write \(\scrJ \subset \Updelta\) for the curves that are contracted  to \(\scrY,\) and so \(\scrJ^{c}=\Updelta\setminus \scrJ\) are those that survive. 
	
	The Dynkin    data  \((\Updelta, \scrJ)\) gives rise to two hyperplane arrangements. One \(\scrH_{\scrJ}\)  is finite and the  other \(\scrH_{\scrJ}^{\aff}\)  is infinite. Both live inside \(\mathbb{R}^{|\Updelta|- |\scrJ|}.\) The hyperplane  arrangement \(\scrH_{\scrJ}\) is calculated by restricting all the positive roots of \(\Updelta\) to the subset \(\scrJ^{c}.\) To calculate    \(\scrH_{\scrJ}^{\aff},\) given  a restricted root \(a=(a_{i})_{i\in \scrJ^{c}},\) the hyperplane \(\sum_{i\in \scrJ^{c}}a_{i}x_{i}=0\) appearing in \(\scrH_{\scrJ}\) gets translated over the set of integers to give an infinite family \(\sum_{i\in \scrJ^{c}}a_{i}x_{i} \in \mathbb{Z}.\) This  is repeated on every restricted root  to give \(\scrH_{\scrJ}^{\aff},\) for full details see e.g \cite{NW,IW7}. Note that when \(\scrJ=\emptyset,\)   that is \(\scrX =\scrY,\) then \(\scrH\) is the finite  $\mathrm{ADE} $ root system and \(\scrH^{\aff}\) is the extended affine root system.
	\begin{example}\label{example of arrangements}
		For  the example of Dynkin data \((\Updelta, \scrJ )\)  with  \(\scrJ= \begin{tikzpicture}
			\draw(0.3,0) circle (2pt);
			\draw (0.5,0) circle (2pt);
			\filldraw (0.7,0) circle (2pt);
			\draw (0.9,0) circle (2pt);
			\draw (1.1,0) circle (2pt);
			\filldraw (0.7,0.2) circle (2pt);
		\end{tikzpicture}, \) where by convention \(\scrJ\)  are the unshaded nodes,  in the diagram below the full figure is the affine hyperplane arrangement  \(\scrH^{\aff}_{\scrJ}.\) Zooming in  to only  the pink lines is  the finite hyperplane arrangement \(\scrH_{\scrJ}.\)
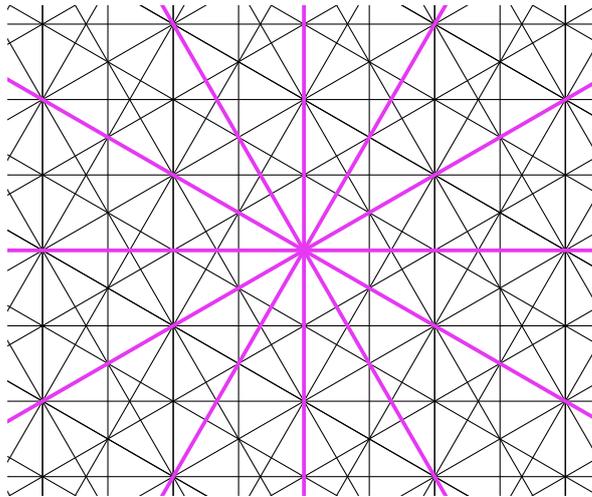
		\[
		\begin{tikzpicture}
			\clip (-4.75,-3.25) rectangle (3,3.25);
			\node at (0,0)
			{
				$\begin{tikzpicture}[xscale=\EeightThreeScalex,yscale=\EeightThreeScaley,
					extended line/.style={shorten >=-#1,shorten <=-#1}]
					\foreach \y in {-2,-1,0,1,2}
					\foreach \x in {-2,0,2}
					{
						\node at (4*\x,4*\y) {$\EeightThreeNoLattice$}; 
					};
					\foreach \y in {-3,-2,-1,0,1}
					\foreach \x in {-3,-1,1,3}
					{
						\node at (4*\x,2+4*\y) {$\EeightThreeNoLattice$};
					};
					\draw[line width=\mythick mm, Pink,extended line=5cm] (0,0) -- (-4,-2);
					\draw[line width=\mythick mm, Pink,extended line=5cm] (0,-1) -- (-4,-1);
					\draw[line width=\mythick mm, Pink,extended line=5cm] (0,-2) -- (-4,0);
					\draw[line width=\mythick mm, Pink,extended line=5cm] (-1,-2.5) -- (-3,0.5);
					\draw[line width=\mythick mm, Pink,extended line=5cm] (-3,-2.5) -- (-1,0.5);
					\draw[line width=\mythick mm, Pink,extended line=5cm] (-2,-3) -- (-2,1);
				\end{tikzpicture}$	};
		\end{tikzpicture}
		\]
	\end{example}
	
	When $\mathcal{H}$ is simplicial, we  write \[ \mathcal{H}_{\mathbb{C}} \colonequals  \bigcup\limits_{H\in \mathcal{H}}  H_{\mathbb{C}}, \] where \(H_{\mathbb{C}}\) is the complexification of a hyperplane \(H.\) We  then denote  the  fundamental group of the complexified complement  by \( \pi_{1}(\mathbb{C}^{n} \backslash \scrH_\mathbb{C})\)  or  \( \pi_{1}(\mathbb{C}^{n} \backslash \scrH^{\aff}_\mathbb{C}),\) for  \(\scrH\) or \(\scrH^{\aff}\) respectively.
	\subsection{ Mutation} As motivation  for finding derived equivalences of a ring, the idea of mutation is to find new tilting modules from old ones, by removing one indecomposable summand of a module and replacing it with another. 
	
	Let \(\scrR\)  be a complete local cDV singularity. A  module \(N\in \smallmod \scrR \) is said to be maximal Cohen--Macaulay ($=\mathrm{CM}$) if 
	\[\mathrm{depth}_{\scrR}N  \colonequals \mathrm{inf} \{i\geq 0 \mid \Ext^{i}_{\scrR}(\scrR/\mathfrak{m},N)\neq 0\}= \mathrm{dim}\scrR,\]  and we write \(\mathrm{CM}~\scrR\) for the category of \(\mathrm{CM}~\scrR\)-modules. Further,  \(N\in \smallmod \scrR \) is said to be reflexive if the natural morphism \(N\to N^{**}\) is an isomorphism, where \((-)^{*} \colonequals \Hom_{\scrR}(-,\scrR),\) and we write \(\refl \scrR\)   for the category of reflexive \( \scrR\)-modules.
	
	\begin{definition}
		A module \(N\in \refl \scrR\)  is called a \emph{modifying} module if \(\End_{\scrR}(N) \in \mathrm{CM}~\scrR,\) and    \(N\in \refl \scrR\)  is called a \emph{maximal modifying} module ($=\mathrm{MM}$) if it is modifying  and 
		\[\mathrm{add}N= \{A\in \refl \scrR \mid \End_{\scrR}(N\oplus A) \in \mathrm{CM}~\scrR\}.\]
		If \(N\) is \(\mathrm{MM},\) then \(\End_{\scrR}(N)\) is called a maximal modification algebra (\(=\mathrm{MMA}\)).
	\end{definition}
	We next summarize mutation, following \cite{IW2}. For  \(C,D \in \smallmod  \scrR,\)   a morphism \(g \colon D_{0} \to C\) is  called  a right (\(\mathrm{add}D\))-approximation if \( D_{0} \in \mathrm{add} D \) and 
	\[\Hom_{\scrR}(D,D_{0}) \xrightarrow{.g} \Hom_{\scrR}(D,C)\]
	is surjective. The  left (\(\mathrm{add}D\))-approximation can be defined dually.  Given  a modifying \(\scrR\)-module \(N,\)  with an indecomposable summand \(N_{i}, \)  consider 
	\begin{enumerate}
		\item  a right $\left(\mathrm{add}\frac{N}{N_{i}}\right)$-approximation of \(N_{i},\) namely  \(V_{i} \xrightarrow{a_{i}} N_{i}\) 
		\item  a right $\left(\mathrm{add}\frac{N}{N_{i}}\right)^{*}$-approximation of \(N^{*}_{i}, \) namely  \(U^{*}_{i} \xrightarrow{b_{i}} N^{*}_{i}\) 
	\end{enumerate}
	which give exchange sequences 
	\begin{equation}\label{intro sequence}
		0 \to \ker a_{i} \to V_{i} \xrightarrow{a_{i}} N_{i} \quad \mbox{and} \quad  0 \to \ker b_{i} \to U^{*}_{i} \xrightarrow{b_{i}} N^{*}_{i}.
	\end{equation}
	From this 
	\begin{enumerate}
		\item   the right mutation of \(N\) at \(N_{i}\) is defined by 
		\[\upnu_{i}(N) \colonequals \frac{N}{N_{i}} \oplus  \ker a_{i},\] i.e., remove the indecomposable  summand \(N_{i}\) and replace it with \(\ker a_{i}.\)
		\item The left  mutation of \(N\) at \(N_{i}\) is defined by 
		\[\upmu_{i}(N) \colonequals \frac{N}{N_{i}} \oplus  (\ker b_{i})^{*}.\]
	\end{enumerate}
	
	As in \cite[Appendix A]{W1},  write \[ \Upphi_{i} \colon \Db(\End_{\scrR}(N)) \to  \Db(\End_{\scrR}(\upnu_{i}N))\] for the associated derived equivalence \(\RHom_{\End_{\scrR}(N)}(\Hom_{\scrR}(N,\upnu_{i}N),-).\)
	\subsection{Affine Auslander-McKay  bijection}\label{affine Auslander }
	We will consider the  general situation of Setup \ref{01 of 23 april 2023} where we  input a flopping contraction \(f \colon U\to \Spec R\) with terminal Gorenstein singularities, with formal fibre \( \scrU \to \Spec \scrR.\) In this case set \(N\colonequals f_{*}\left(\mathcal{O}_{\scrU} \oplus \mathcal{N}\right), \)  where \(\mathcal{O}_{\scrU} \oplus \mathcal{N}\) is the \cite{VdB} tilting bundle on \(\scrU.\) Complete locally we have the following diagram,
	\[
	{\scriptsize
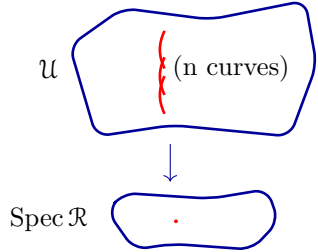
		\begin{tikzpicture}
			\node at (-1.9,0) {$\scrU$};
			\node at (0.5,0) {$\mathrm{(n~curves)}$};
			\node at (-1.9,-2) {$\Spec\scrR$};
			\node at (0,0) {\begin{tikzpicture}[scale=0.5]
					\coordinate (T0) at (1.8,3);
					\coordinate (B0) at (1.8,2);
					\coordinate (T) at (1.8,2.3);
					\coordinate (B) at (1.8,1.3);
					\coordinate (T1) at (1.8,1.8);
					\coordinate (B1) at (1.8,0.8);
					\draw[red,line width=1pt] (B) to [bend left=25] (T);
					\draw[red,line width=1pt] (B1) to [bend left=25] (T1);
					\draw[red,line width=1pt] (B0) to [bend left=25] (T0);
					\draw[color=blue!60!black,rounded corners=5pt,line width=1pt] (0,0.2) -- (2,0.5)-- (5.1,0.2) -- (5.8,2.1)-- (5.5,3.8) -- (3,3.3) -- (0.7,3.6) -- (-0.7,2.6)-- cycle;				
			\end{tikzpicture}};
			\node at (0,-2) {\begin{tikzpicture}[scale=0.5]
					\filldraw [red] (2.1,0.75) circle (1pt);
					\draw[color=blue!60!black,rounded corners=5pt,line width=1pt] (1,0) -- (2,0.15)-- (4.1,0) -- (4.8,0.75)-- (4.5,1.6) -- (3,1.35) -- (0.7,1.5) -- (0.3,1)-- cycle;
			\end{tikzpicture}};
			\draw[->, color=blue!60!black] (-0.3,-1) -- (-0.3,-1.5);
	\end{tikzpicture}}
	\]
	
	Using the same notation as in \cite{IW7}, we will write  \(\mathrm{MM}^{N}\scrR\) for those modifying reflexive \(\scrR\)-modules that have a two-term approximation by \(\mathrm{add} N\) and have the same number of indecomposable summands as \(N\). Further, write   \(\mathrm{MMG}^{N}\scrR\) for those in  \(\mathrm{MM}^{N}\scrR\) which have \(\scrR\) as a summand. By the general version of affine  Auslander-McKay correspondence see e.g \cite[0.18]{IW7},  \(\mathrm{MM}^{N}\scrR\) coincides  with the mutation classes of \(N\) and \(\mathrm{MMG}^{N}\scrR\) coincides with the Cohen-Macaulay  mutation classes of \(N,\)
	and there are the following bijections \[
	\begin{tikzpicture}
		\node (A) at (1.5,0) {$\{ \mbox{MM}^{N}\scrR\}$};
		\node (B) at (5.7,0) {$\{\mbox{Chambers of $\scrH^{\aff}$} \}$};
		\node (C) at (1.5,-1) {$\{ \mbox{MMG}^{N}\scrR\}$};
		\node (D) at (5.5,-1) {$\{\mbox{Chambers of $\scrH$} \}$};
		\draw[<->] (2.5,0) -- node [above] {$\scriptstyle  1\colon 1$} (4,0);
		\draw[<->] (2.5,-1) -- node [above] {$\scriptstyle  1\colon 1$} (4,-1);
		\draw (1.2,-1) [transparent]edge node[rotate=90,opacity=1] {$\subseteq$} (1.7,0);
		\draw (8, -1) [transparent]edge node[rotate=90,opacity=1] {$\subseteq$} (2,0);
	\end{tikzpicture}
	\] under which  wall crossing corresponds to mutation. As notation, write \(N_{D}\) for the element in \(\mathrm{MM}^{N}\scrR\) corresponding to chamber \(D.\)
	
	\begin{example} 
		Continuing the Example \ref{example of arrangements}, the following green chamber (or alcove) corresponds  to $N$ and  has 3 summands (corresponding to the three walls).  We can mutate  to obtain all elements of  \(\mathrm{MM}^{N}\scrR\)  by beginning in the green chamber.
		\[ {\scriptsize
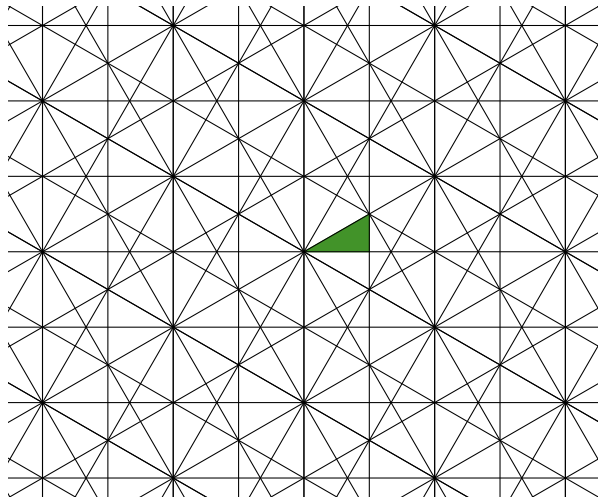
			\begin{tikzpicture}
				\clip (-4.75,-3.25) rectangle (3,3.25);
				\node at (0,0)
				{
					$\begin{tikzpicture}[xscale=\EeightThreeScalex,yscale=\EeightThreeScaley]
						\filldraw[Green] (0,0) --($(0,0)!0.5!(0,-2)$)--(-2,-1)--cycle;
						\foreach \y in {-2,-1,0,1,2}
						\foreach \x in {-2,0,2}
						{
							\node at (4*\x,4*\y) {$\EeightThreeNoLattice$}; 
						};
						\foreach \y in {-3,-2,-1,0,1}
						\foreach \x in {-3,-1,1,3}
						{
							\node at (4*\x,2+4*\y) {$\EeightThreeNoLattice$};
						};
					\end{tikzpicture}$
				};
		\end{tikzpicture} }
		\]
		
		Bounding the above figure  to following the pink  box  illustrates the  chambers in the finite arrangement \(\scrH.\) We  can obtain all elements of  \(\mathrm{MMG}^{N}\scrR\) (which correspond to the chambers within the pink box)   by  repeatedly mutating all summands except \(\scrR.\) 
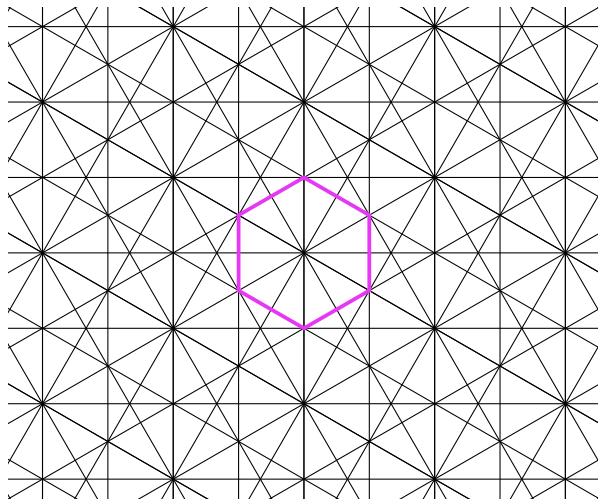
		\[
		\begin{array}{c}
			\begin{tikzpicture}
				\clip (-4.75,-3.25) rectangle (3,3.25);
				\node at (0,0)
				{
					$\begin{tikzpicture}[xscale=\EeightThreeScalex,yscale=\EeightThreeScaley]
						\foreach \y in {-2,-1,0,1,2}
						\foreach \x in {-2,0,2}
						{
							\node at (4*\x,4*\y) {$\EeightThreeNoLattice$}; 
						};
						\foreach \y in {-3,-2,-1,0,1}
						\foreach \x in {-3,-1,1,3}
						{
							\node at (4*\x,2+4*\y) {$\EeightThreeNoLattice$};
						};
						\draw[line width=\mythick mm, Pink] (0,0) --(0,-2)--(-2,-3)--(-4,-2)--(-4,0)--(-2,1)--cycle;
					\end{tikzpicture}$
				};			
			\end{tikzpicture}
		\end{array}
		\]
	\end{example}
	
	\section{Preliminary Lemmas}\label{preliminary}
	Returning to Setup \ref{01 of 23 april 2023}, consider the flopping contraction \(U\to \Spec R \) with formal fibre \(\scrU \to \Spec \scrR.\) As explained in \S \ref{geometric setup},  \( \Lambda=\End_{R}(R\oplus  f_{*}\mathcal{N}), \)  thus after completion \[\scrR \oplus  \widehat{f_{*}\mathcal{N}}=\scrR^{\oplus a_{0}} \oplus  N_{1}^{\oplus a_{1}} \oplus \hdots \oplus N_{n}^{\oplus a_{n}}, \] for some \(a_{i} \in \mathbb{N},\) and so 
	\[\hat{\Lambda} = \End_{\scrR} (\scrR^{\oplus a_{0}} \oplus  N^{\oplus a_{1}} \oplus \hdots \oplus N^{\oplus a_{n}}).\] Set  $\Anohat=\End_\scrR(N)$ with $N=\scrR\oplus N_1\oplus\hdots\oplus N_n,$ and write that $\scrR\colonequals   \hat{R}_{\mathfrak{m}} $.

	\begin{setup}\label{bi setup}
		Given a choice of an atom  \(\upalpha\colon C_{+} \to D\) in \(\scrH \) or \(\scrH^{\aff}\)  and any wall \(i\) of \(D\) we next associate a functor to \(\Db(\Lambda).\) 
		Note that \(\Anohat\) is the basic algebra Morita equivalent to \(\hat{\Lambda}\), by e.g.,\ \cite [2.6]{DW3}. Consider $\Bnohat\colonequals \End_\scrR(N_D),$ where \(N_{D}\) is described in \(\S\ref{affine Auslander }.\) Now, $\Bnohat$  inherits  \(n+1\) primitive idempotents  \(e_{i}\)  corresponding to   the  walls  of \(D.\)  For purposes of notation,  write \(\mathbb{F}\) for Morita equivalence,  \(\Upphi_{\upalpha}\) for composition of  mutation functors  corresponding to  the path \(\upalpha,\)  and set $ \Bnohat_i\colonequals \Bnohat/({1-e_i}).$  Consider the composition		
		\begin{equation}\label{eqn4}
			\Uppsi_{\upalpha,i} \colonequals \D(\Bnohat_i) \xrightarrow{}\D(\Bnohat) \xrightarrow{\Upphi_{\upalpha}} \D(\Anohat)\xrightarrow{\sim}\D(\hat{\Lambda}) \xrightarrow{}\D(\Lambda) 
		\end{equation}		
		where the first is induced by the ring homomorphism $\Bnohat\to\Bnohat_i$, the second is the composition mutations, the third is the Morita equivalence as in \cite[5.4]{DW1} and above, and the fourth is induced by the ring homomorphism $\Lambda\to\hat{\Lambda}$.
	\end{setup} 
	Thus, given information of \((\upalpha,i),\) which is complete local since \(\scrH\) and \(\scrH^{\aff}\) are constructed complete locally, we have obtained a functor \(\Uppsi_{\upalpha,i} \) between the derived categories of algebraic objects.
	
	Given that \(\Upphi_{\upalpha} \) is the composition of standard equivalences, we may write \(\Upphi_{\upalpha} \cong \RHom_{\Bnohat}({}_{\Anohat}\uptau_{\upalpha},-)\) for some \(\Aphat\)-\(\Bnohat\) bimodule \(\uptau_{\upalpha}.\) Similarly, write \(\Hom_{\Aphat}({}_{\hat{\Lambda}}\mathbb{M},-)\) for the Morita equivalence between \(\Anohat\) and \(\hat{\Lambda}.\) We have the following result.
	\begin{lemma}\label{01 of 4th november 2022}
		There is a functorial isomorphism
		\begin{equation}\label{eqn5}\Uppsi_{\upalpha,i} \cong 
			\begin{cases}
				\! 
				\begin{alignedat}{1}
					&\RHom_{\Lambda}\left( \hat{\Lambda}\otimes^{\bf{L}}_{\hat{\Lambda}} \mathbb{M}\otimes^{\bf{L}}_{\Aphat} \uptau_{\upalpha}\otimes _{\Bphat}^{\bf{L}}\Bnohat_{i},-\right) \\
					&-\otimes^{\bf{L}}_{\Bnohat_{i}}\left( \Bnohat_{i}\otimes _{\Bphat}^{\bf{L}} \,\uptau^{*}_{\upalpha} \otimes^{\bf{L}}_{\Aphat} \mathbb{M}^{*}\otimes^{\bf{L}}_{\hat{\Lambda}}\hat{\Lambda} \right) 							
				\end{alignedat}
				
			\end{cases}
		\end{equation}
	\end{lemma}
	\begin{proof}
		Consider the functor diagram
		\[
		\D(\Bnohat_{i}) \xrightarrow[-\otimes^{\bf{L}}_{\Bphat_{i}}{\Bnohat_{i}}_{\Bnohat}]{\RHom_{\Bnohat_{i}}({}_{\Bnohat}\!\Bnohat_{i},-)}
		\D(\Bnohat) 
		\xrightarrow[-\otimes^{\bf{L}}_{\Bphat} {\,\uptau^{*}_{\upalpha}}_{\Aphat}]{\RHom_{\Bnohat}({}_{\Anohat}\uptau_{\upalpha},-)}
		\D(\Anohat) \xrightarrow[-\otimes^{\phantom{L}}_{\!\Aphat} {\mathbb{M}_{\hat{\Lambda}}^{*}}]{\Hom_{\Aphat}({}_{\hat{\Lambda}}\mathbb{M},-)} 
		\D(\hat{\Lambda})\xrightarrow[-\otimes^{\bf{L}}_{\hat{\Lambda}} \,\hat{\Lambda}^{\phantom{L}}_{\Lambda}]{\RHom_{\hat{\Lambda}}({}^{\phantom{L}}_{\Lambda}\!\hat{\Lambda},-)}\D(\Lambda).
		\] 
		The left most and right most functors are restriction of scalars, and  so in both cases the top functor equals the bottom one. The middle functors are equivalences, so again  in each case, the top functor is isomorphic to the bottom functor,  now by e.g \cite[p.44]{R3}.  The statement follows since \eqref{eqn5} is the composition.				
	\end{proof}
	
	\begin {lemma}\label{02 of 05th november 2022}
	\(\uprho \colonequals \left( \Bnohat_{i}\otimes_{\Bphat}^{\bf{L}} \,\uptau^{*}_{\upalpha} \otimes^{\bf{L}}_{\Aphat} \mathbb{M}^{*}\otimes^{\bf{L}}_{\hat{\Lambda}}\hat{\Lambda} \right) \cong  \mathbb{F}\Upphi_{\upalpha}\left(\Bnohat_{i}\right) \) as  right \(\Lambda\)-modules.
\end{lemma}
\begin{proof}
	Observe that \(\Bnohat_{i}\mapsto \Bnohat_{i}\mapsto \Upphi_{\upalpha}(\Bnohat_{i})\mapsto \mathbb{F}\Upphi_{\upalpha}(\Bnohat_{i}) \mapsto \mathbb{F}\Upphi_{\upalpha}(\Bnohat_{i})\) under the composition defining \(\Uppsi_{\upalpha,i}\). Thus
	\[
	\mathbb{F}\Upphi_{\upalpha}(\Bnohat_{i})\cong \Uppsi_{\upalpha,i}(\Bnohat_{i}) \stackrel{\ref{01 of 4th november 2022}} {\cong } \Bnohat_{i}\otimes^{\bf{L}}_{\Bphat_{i}}\uprho \cong \uprho\] as right \(\Lambda\)-modules.
\end{proof}
\begin{claim}\label{01 0f 20th november 2022}
	\(\mathbb{F}\Upphi_{\upalpha}(S_{i})\) is  a finite length module  supported only  at the maximal ideal \(\mathfrak{m}\).
\end{claim}
\begin{proof}
	Since \(\upalpha\) is an atom, by torsion pairs (see e.g \cite[\(\S5\)]{HW}), \(\Upphi_{\upalpha}(S_{i})\) is a module or a shift of a module, in particular \(\Upphi_{\upalpha}(S_{i})\) is in a single homological degree. Thus, since \(\mathbb{F}\) is a Morita equivalence, \(\mathbb{F}\Upphi_{\upalpha}(S_{i})\) is also in a single homological degree. Since \(\Bnohat_{i}\) is a finite dimensional algebra filtered by the simple \(S_{i}\), then  \(\mathbb{F}\Upphi_{\upalpha}\left(\Bnohat_{i}\right) \) is filtered by  \(\mathbb{F}\Upphi_{\upalpha}(S_{i}).\)  It is known that \(S_{i}\) is supported  only at the maximal ideal \(\m\), thus for $\p\neq\m$
	\[
	\Upphi_{\upalpha}(S_{i})_{\p}=\mathbf{R}\text{Hom}_{\Bnohat}(\uptau_{\upalpha},S_{i})_{\p}
	=\mathbf{R}\text{Hom}_{\Bnohat_{\p}}({\uptau_{\upalpha}}_\p,{S_{i}}_{\p})
	=0.
	\]		
	So, \(\mathbb{F}\Upphi_{\upalpha}(S_{i})=\Upphi_{\upalpha}(S_{i})\otimes^{\phantom{L}}_{\!\Anohat}\! {\mathbb{M}^{*}}\) is also supported at only  the maximal ideal \(\mathfrak{m}.\) Thus, \(\mathbb{F}\Upphi_{\upalpha}(S_{i})_{\hat{\Lambda}}\in \fl\hat{\Lambda}\). Since by Iyama--Reiten \cite{IR} \(\fl\hat{\Lambda} \hookrightarrow \fl\Lambda \), then \(\mathbb{F}\Upphi_{\upalpha}(S_{i})\) is  a finite length \(\Lambda\)-module supported only at the maximal ideal \(\mathfrak{m}.\)
\end{proof}
Recall from  $\S$\ref{specific hyperplanes} the associated hyperplane arrangements \(\scrH\) and \(\scrH^{\aff}.\)
\begin{lemma}\label{02 of 24 january 2023}
	If \(\upalpha \) is an atom in \(\scrH\) or \( \scrH^{\aff}\), then \(\mathrm{H}^{t}(\Upphi^{-1}_{\upalpha}(\Bnohat_{i}))=0\) for all but one \(t.\)
\end{lemma}
\begin{proof}
	\(\Bnohat_{i}\) is filtered by the simple \(S_{i}\). By the torsion pairs described in \cite{HW}, since \(\upalpha\) is an atom we know that \(\Upphi^{-1}_{\upalpha}(S_{i})\) is only in one homological degree. Thus \(\Upphi_{\upalpha}^{-1}(\Bnohat_{i}) \) is in one homological degree only.
\end{proof}
The following is also well known.
\begin{lemma}\label{04 0f 24 january 2023}
	With notation as above, the following statements hold.
	\begin{enumerate}
		\item\label{04 0f 24 january 2023 one} If \(Z\)  is a \(\fl \hat{\Lambda}\)-module, then the natural map  \(Z\otimes_{\hat{\Lambda}} \hat{\Lambda}\otimes_{\Lambdaphat} \hat{\Lambda} \to Z \) is an isomorphism of \(\hat{\Lambda}\)-modules.
		\item\label{04 0f 24 january 2023 two}	If \(a\in \Dbfl(\smallmod \hat{\Lambda} ), \) then \(a\otimes_{\hat{\Lambda}} \hat{\Lambda}\otimes_{\Lambdaphat} \hat{\Lambda} \cong a\) in \(\Dbfl( \smallmod \hat{\Lambda})\).
	\end{enumerate}
	
\end{lemma}
\begin{proof}
	(1) Consider the following  functors between categories  
	\[  \begin{tikzcd}
		\fl \Lambda \ar[r, shift left=1ex, "\otimes_{\Lambdaphat}\Lambda_{\mathfrak{m}_{\Lambda_{\mathfrak{m}}}}"] &	\fl \Lambda_{\mathfrak{m}} \ar[r,shift left=1ex,"\otimes_{\Lambdaphat_{\mathfrak{m}}}\hat{\Lambda}_{\hat{\Lambda}}"]  \arrow[l,   hook', shift left= 1ex, "\otimes_{\Lambdaphat_{\mathfrak{m}}}\Lambda_{\mathfrak{m}_{\Lambda}}" ]& \fl \hat{\Lambda}  \ar[l, shift left=1ex, "\otimes_{\hat{\Lambda}} \hat{\Lambda}_{\Lambdaphat_{\mathfrak{m}}}"],
	\end{tikzcd} \]  
	where the bottom  left  functor is fully  faithful by \cite{IR}. By \cite[2.15]{DW1}, \cite[p1100]{IR}, the rightmost functors are an equivalence of categories, thus composing we have  
	\[
	\begin{tikzcd}
		\fl \Lambda \ar[r, shift left=1ex] & \fl \hat{\Lambda}  \arrow[l, shift left=1ex,  hook', "\otimes_{\hat{\Lambda}}\hat{\Lambda}_{\Lambdaphat}" ].
	\end{tikzcd} 
	\] 
	The claimed morphism is the counit, and this is an isomorphism since the right adjoint functor \( \otimes_{\hat{\Lambda}}\hat{\Lambda}_{\Lambdaphat}\) is fully faithful.\\
	(2)  By  \cite[2.5]{IR} there is an equivalence of categories \(\Db(\fl \hat{\Lambda}) \xrightarrow{\sim } \Dbfl(\smallmod \hat{\Lambda})\), so  
	\[
	a\cong \cdots \to a_{i+1} \to a_{i}\to a_{i-1} \to \cdots
	\] 
	with each \(a_{i}\) a finite length \(\hat{\Lambda}\)-module. Since by  Lemma \ref{04 0f 24 january 2023}\eqref{04 0f 24 january 2023 one} \(a_{i}\otimes_{\hat{\Lambda}} \hat{\Lambda} \otimes_{\Lambdaphat} \hat{\Lambda} \cong a_{i}\)  via the counit,  it follows that \(a \cong a \otimes_{\hat{\Lambda}} \hat{\Lambda} \otimes_{\Lambdaphat} \hat{\Lambda}\)  as complexes, as required.
\end{proof}
\section{ Construction of Bimodule Cones }\label{twistfun}
In the local Zariski setting of \ref{01 of 23 april 2023} where \(\Lambda\) is derived equivalent to \(U,\) this section uses bimodule cones to construct endofunctors of \(U.\) 
\begin{prop}\label{01 of  06 December 2022} 
	For any choice of \((\upalpha,i)\) as in Setup \textnormal{\ref{bi setup}}, there exists \({}_{\Lambda}C_{\Lambda}=C_{\upalpha,i}\) in the derived category of \(\Lambda\)-\(\Lambda\) bimodules, such that, given  the functorial diagram
	\begin{eqnarray}\label{bimodules1}
		\begin{array}{c}
			\begin{tikzpicture}[looseness=1,bend angle=20]
				\node (a1) at (0,0) {$\D(\Bnohat_{i})$};
				\node (a2) at (3,0) {$\D( \Lambda)$};
				\draw[->] (a1) -- node[gap] {$\scriptstyle \Uppsi_{\upalpha,i}$} (a2);
				\draw[->,bend right] (a2) to node[above] {$\scriptstyle \Uppsi_{\upalpha,i}^{\LA}$} (a1);
				\draw[->,bend left] (a2) to node[below] {$\scriptstyle \Uppsi_{\upalpha,i}^{\RA}$} (a1);										
			\end{tikzpicture}
		\end{array},
	\end{eqnarray}
	setting \(\TW \colonequals\RHom_{\Lambda} (C_{\Lambda},-)\) and   \(\TW^{*} \colonequals -\otimes^{\bf{L}}_{\Lambda}C_{\Lambda},\) there are  functorial triangles 
	\begin{enumerate}
		\item \(\Uppsi_{\upalpha,i}\circ  \Uppsi_{\upalpha,i}^{\RA} \to \Id_{\Lambda}\to \TW_{\upalpha,i} \to\)
		\item \( \TW^{*}_{\upalpha,i} \to \Id_{\Lambda}\to \Uppsi_{\upalpha,i}\circ  \Uppsi_{\upalpha,i}^{\LA} \to.\)	\end{enumerate}
\end{prop}
\begin{proof}
	(1)  We expand  the diagram in \eqref{bimodules1}  to 
	\[
	\begin{tikzcd}[sep=normal,cramped]
		\D(\Bnohat_{i}) \arrow[r,"\mathrm{H}"] 
		& \D(\Bnohat) \arrow[r,"\Upphi_{\upalpha}"] \arrow[l, bend left=60, ,"\mathrm{H}^{\RA}"] &	\D\left(\Anohat\right) \arrow[l, bend left=50, ,"\Upphi_{\upalpha}^{\RA}"] 
		\arrow[r,"\mathbb{F}"] 
		& \D(\hat{\Lambda})\arrow[r,"\mathrm{F}"]\arrow[l, bend left=50, ,"\mathbb{F}^{\RA}"] 	& \D(\Lambda)\arrow[l, bend left=50, ,"\mathrm{F}^{\RA}"] 					
	\end{tikzcd}\] 
	where
	\[
	\begin{aligned}
		\mathrm{H}= \RHom_{\Bnohat_{i}}( {}_{\Bnohat} {\Bnohat_{i}},-) &\cong -\otimes^{\bf{L}}_{\Bphat_{i}}\!{\Bnohat_{i}}_{\Bnohat}  &\qquad  \mathrm{H}^{\RA} &=\RHom_{\Bnohat}( {}_{\Bnohat_i}{\Bnohat_{i}},-)\\
		\Upphi_{\upalpha}= \RHom_{\Bnohat}( {}_{\Anohat} \uptau_{\upalpha},-) &\cong -\otimes^{\bf{L}}_{\Bphat} {\uptau^{*}_{\upalpha}}_{\Anohat}  &\qquad  \Upphi_{\upalpha}^{\RA} &=\RHom_{\Anohat}({}_{\Bnohat}\uptau^{*}_{\upalpha},-)\\
		\mathbb{F}=\RHom_{\Anohat}({}_{\hat{\Lambda}} \mathbb{M},-) &\cong -\otimes^{\bf{L}}_{\Aphat}\mathbb{M}^{*}_{\hat{\Lambda}}
		&\qquad  \mathbb{F}^{\RA} &= \mathbf{R}\text{Hom}_{\hat{\Lambda}}({} _{\Aphat}\mathbb{M}^{*},-)\\
		\mathrm{F}=\RHom_{\hat{\Lambda}}( {}_{\Lambda} \hat{\Lambda},-) &\cong -\otimes^{\bf{L}}_{\hat{\Lambda}}\hat{\Lambda}_{\Lambda}
		&\qquad  \mathrm{F}^{\RA} &=\mathbf{R}\text{Hom}_{\Lambdaphat}( {}_{\hat{\Lambda} }\hat{\Lambda},-)		
	\end{aligned}
	\]
	
	We will construct \(\Uppsi_{\upalpha,i}\circ \Uppsi_{\upalpha,i}^{\RA}\to \Id_{\Lambda}\) by exhibiting  a certain bimodule map. First, the ring homomorphism \(\Bnohat\to \Bnohat_{i}\) is a map of $\Bnohat$-$\Bnohat$ bimodules and thus induces a natural transformation $\RHom_{\Bnohat}( {}_{\Bnohat}{\Bnohat_{i}}_{\Bnohat},-)\to \RHom_{\Bnohat}(\Bnohat, -)$.  But this is equal to
	\begin{align*}
		&=\RHom_{\Bnohat}({}_{\Bnohat}\!\Bnohat_{i} \otimes^{\bf{L}}_{\Bphat_i}  {\Bnohat_{i}}_{\Bnohat},-)\to \Id_{\D(\Bnohat)}\\ 
		&\cong\mathrm{H}\circ \mathrm{H}^{\RA} \to \Id_{\D(\Bnohat)}						
	\end{align*}
	
	We next construct a natural transformation  \( \Upphi_{\upalpha }\circ \mathrm{H} \circ \mathrm{H}^{\RA}\circ \Upphi_{\upalpha}^{\RA} \to \Id_{\D(A)}.\) Tensoring \(\Bnohat\to \Bnohat_{i}\) on the left by  \(\uptau_{\upalpha}\) and on the right by  \(\uptau^{*}_{\upalpha}\) gives rise to a bimodule map 
	\[ 
	\Anohat \xrightarrow{\sim} 
	\uptau_{\upalpha}\otimes^{\bf{L}}_{\Bphat}\uptau^{*}_{\upalpha}\cong 
	\uptau_{\upalpha}\otimes^{\bf{L}}_{\Bphat}\Bnohat\otimes^{\bf{L}}_{\Bphat}\,\uptau^{*}_{\upalpha}
	\to 
	\uptau_{\upalpha}\otimes^{\bf{L}}_{\Bphat} \Bnohat_{i}\otimes^{\bf{L}}_{\Bphat} \,\uptau^{*}_{\upalpha} \tag{the first map is by \cite[4.2]{R3}}  
	\] 
	which induces the claimed \( \Upphi_{\upalpha }\circ \mathrm{H} \circ \mathrm{H}^{\RA}\circ \Upphi_{\upalpha}^{\RA} \to \Id_{\D(A)}.\)
	
	Now, applying the same trick for \(\mathbb{F},\) setting  \(Z_{\upalpha,i}\colonequals \Bnohat_{i}\otimes_{\Bphat}^{\bf{L}} \,\uptau_{\upalpha}^{*}\otimes_{\Aphat}^{\bf{L}} \mathbb{M}^{*},\) then  \begin{equation}{\scriptsize\label{01 of 5th december 2022}
			\begin{array}{c}
				\begin{tikzpicture}[looseness=1,bend angle=10]
					\node (a1) at (0,0) {$\hat{\Lambda}$};
					\node (a2) at (3,0) {$ \mathbb{M}\otimes_{\Aphat}^{\bf{L}}\mathbb{M}^{*}\cong \mathbb{M}\otimes_{\Aphat}^{\bf{L}}\!\Anohat \otimes_{\Aphat}^{\bf{L}}\, \mathbb{M}^{*} $};
					\node (a3) at (9,0) {$ \mathbb{M}\otimes^{\bf{L}}_{\Aphat} \left( \uptau_{\upalpha}\otimes^{\bf{L}}_{\Bphat} \Bnohat_{i}\otimes^{\bf{L}}_{\Bphat} \uptau^{*}_{\upalpha}\right) \otimes^{\bf{L}}_{\Aphat} \mathbb{M}^{*}  \cong \mathbb{M}\otimes^{\bf{L}}_{\Aphat}  \uptau_{\upalpha}\otimes^{\bf{L}}_{\Bphat} Z_{\upalpha,i} $};
					\draw[->] (a1)-- node [above]{$\sim$} (a2);
					\draw[->] (a2) -- node []{$ $} (a3);
					\draw[ ->, bend right] --(a1) to node[below] {$\scriptstyle h$} (a3);
		\end{tikzpicture}\end{array}}
	\end{equation}
	which induces   \(\mathbb{F}\circ \Upphi_{\upalpha}\circ \mathrm{H} \circ \mathrm{H}^{\RA}\circ \Upphi_{\upalpha}^{\RA} \circ \mathbb{F}^{\RA}\to \Id_{\D(\hat{\Lambda})}.\)  
	
	Lastly, tensoring \eqref{01 of 5th december 2022} on both sides by \(\hat{\Lambda}\) gives a \(\hat{\Lambda}\)-\(\hat{\Lambda}\) bimodule map
	\begin{eqnarray}\label{01 of 25 january 2023}
		\begin{array}{c}
			\begin{tikzpicture}[looseness=1,bend angle=10]
				\node (a1) at (0.5,0) {$\hat{\Lambda}$};
				\node (a2) at (3,0) {$ \hat{\Lambda}\otimes^{\bf{L}}_{\hat{\Lambda}} \hat{\Lambda}\otimes^{\bf{L}}_{\hat{\Lambda}} \hat{\Lambda}$};
				\node (a3) at (8.3,0) {$ \hat{\Lambda} \otimes^{\bf{L}}_{\hat{\Lambda}} \left(  \mathbb{M}\otimes^{\bf{L}}_{\Aphat}  \uptau_{\upalpha}\otimes^{\bf{L}}_{\Bphat}Z_{\upalpha,i} \right)\otimes^{\bf{L}}_{\hat{\Lambda}} \hat{\Lambda}  $};
				\draw[->] (a1)-- node [above]{$ \scriptstyle \sim$} (a2);
				\draw[->] (a2) -- node [above]{$ \scriptstyle 1\otimes h\otimes 1 $} (a3);
				\draw[ ->, bend right] --(a1) to node[below] {$\scriptstyle f$} (a3);
		\end{tikzpicture}\end{array}
	\end{eqnarray}
	Composing this with the ring homomorphism \(\Lambda \to \hat{\Lambda}\) thus  gives a bimodule map 
	\[
	\Lambda \to {}_{\Lambda}\hat{\Lambda} \otimes \left(  \mathbb{M}\otimes  \uptau_{\upalpha}\otimes Z_{\upalpha,i}\right)\otimes\hat{\Lambda}_{\Lambda}.
	\]					
	Taking the cone in the derived category of \(\Lambda\)-\(\Lambda\) bimodules gives a triangle 
	\begin{equation}\label{01 of 24 january 2023}
		C \to {}_{\Lambda}\Lambda_{\Lambda} \to {}_{\Lambda}\hat{\Lambda} \otimes^{\bf{L}}_{\hat{\Lambda}} \left(  \mathbb{M}\otimes^{\bf{L}}_{\Aphat}  \uptau_{\upalpha}\otimes^{\bf{L}}_{\Bphat} Z_{\upalpha,i} \right)\otimes^{\bf{L}}_{\hat{\Lambda}} \hat{\Lambda}_{\Lambda}\to  C[1].
	\end{equation} 
	This induces a functorial triangle \(\Uppsi_{\upalpha,i}\circ  \Uppsi_{\upalpha,i}^{\RA} \to \Id_{\Lambda}\to \RHom ({}_{\Lambda}C_{\Lambda},-)\to.\) Defining \(\TW_{\upalpha,i}\colonequals \RHom ({}_{\Lambda}C_{\Lambda},-),\) yields the result.\\
	(2)  \eqref{01 of 24 january 2023} also induces a functorial triangle \(-\otimes^{\bf{L}}_{\Lambda} C_{\Lambda}\to \Id \to \Uppsi_{\upalpha,i}\circ \Uppsi_{\upalpha,i}^{\LA}\to, \) thus defining \(\TW_{\upalpha,i}^{*}\colonequals -\otimes^{\bf{L}}_{\Lambda} C_{\Lambda}\) gives the claim.
\end{proof}

\begin{remark}
	In the proof of Proposition \ref{01 of  06 December 2022},  \(Z_{\upalpha,i}\colonequals \Bnohat_{i}\otimes_{\Bphat}^{\bf{L}} \,\uptau_{\upalpha}^{*}\otimes_{\Aphat}^{\bf{L}} \mathbb{M}^{*}\).  We will use this notation below.
\end{remark}
\begin{remark} \label{01 of 21 feb 2023}
	\(\TW_{\upalpha,i}^{*}\) is the left adjoint of \(\TW_{\upalpha,i}.\) Later, in Remark \ref{inverse of twist} we will prove that it is the inverse.
\end{remark}
\section{The Key Commutative Diagram}\label{section Key Commutative Diagram}
Under the Zariski local Setup \ref{01 of 23 april 2023}   and  given any choice of \((\upalpha,i)\) as defined in Setup \ref{bi setup}, this section proves that the functors \(\TW_{\upalpha,i}\) and \(\TW^{*}_{\upalpha,i}\) constructed in $\S$\ref{twistfun} intertwine with the known equivalences in the complete local setting,    through a key commutative diagram.

Given a choice of \((\upalpha,i),\) in particular with associated homomorphism  \(\Bnohat\to \Bnohat_{i},\) set \(I_{i}\) to be the kernel  of the homomorphism, which is a two-sided ideal of \(\Bnohat.\)
\begin{claim}\label{infinite Ii}
	Given an atom \(\upalpha\colon C_{+} \to D\) and a wall \(i\) of \(D,\)	set \(M=N_{D} \) so that \(\Bnohat\colonequals \End_{\scrR}(M) \) and  \(   \upnu_{\kern-1pt i}\kern-2pt \Bnohat=\End_{\scrR}(\upnu_{i}M).\) Then \[I_{i}\cong \Hom_{\scrR}(M, \upnu_{i}M) \otimes^{\bf{L}}_{ \upnu_{\kern-1pt i}\kern-2pt \Bnohat} \Hom_{\scrR}( \upnu_{i}M,M) \] as \(\Bnohat\)-bimodules.
\end{claim}
\begin{proof}
	The proof of \cite[5.10(1)]{DW1} shows \[I_{i}\cong \Hom_{\scrR}(M, \upnu_{i}M) \otimes^{\bf{L}}_{   \upnu_{\kern-1pt i}\kern-2pt \Bnohat} \Hom_{\scrR}( \upnu_{i}M,M) \] as right modules by showing that \(\Hom_{\scrR}(M, \upnu_{i}M) \otimes^{\bf{L}}_{   \upnu_{\kern-1pt i}\kern-2pt \Bnohat} \Hom_{\scrR}( \upnu_{i}M,M)\)  is concentrated in degree zero. Hence truncating in the category of bimodules there is an isomorphism of  \(\Bnohat\)-bimodules
	\[\Hom_{\scrR}(M, \upnu_{i}M) \otimes^{\bf{L}}_{   \upnu_{\kern-1pt i}\kern-2pt \Bnohat} \Hom_{\scrR}( \upnu_{i}M,M) \cong \Hom_{\scrR}(M, \upnu_{i}M) \otimes_{   \upnu_{\kern-1pt i}\kern-2pt \Bnohat} \Hom_{\scrR}( \upnu_{i}M,M),  \] so it suffices to show that 
	\[\Hom_{\scrR}(M, \upnu_{i}M) \otimes_{   \upnu_{\kern-1pt i}\kern-2pt \Bnohat} \Hom_{\scrR}( \upnu_{i}M,M)\cong \Bnohat e_{i} \Bnohat\] as \(B\)-bimodules.
	
	Adding \(\frac{M}{M_{i}} \xrightarrow{\Id}\frac{M}{M_{i}}\)  appropriately to the exchange sequence \eqref{intro sequence}, there exists an exact sequence
	\begin{equation}\label{exchange sequence}
		0\to \upnu_{i}M \to V\xrightarrow{b} M,
	\end{equation} where \(b\) is an   \(\left(\mathrm{add}\frac{M}{M_{i}} \right)\)-approximation  of \(M\). We now claim that applying \(\Hom_{\scrR}(\upnu_{i}M,-)\) to \eqref{exchange sequence} gives an exact sequence 
	\begin{equation}\label{exact sequence}
		0\to \Hom_{\scrR}(\upnu_{i}M, \upnu_{i}M) \to  \Hom_{\scrR}(\upnu_{i}M, V)  \to  \Hom_{\scrR}(\upnu_{i}M, M) \to 0.
	\end{equation}  To see this, set \(\Gamma=\End_{\scrR}\left( \frac{M}{M_{i}}\right)\) and  \(\mathbb{G}=\Hom_{\scrR}\left( \frac{M}{M_{i}},-\right)\), then since  \(b\) is an \(\left(\mathrm{add}\frac{M}{M_{i}} \right)\)-approximation  of \(M,\) \[0\to \mathbb{G}(\upnu_{i}M) \to \mathbb{G}V \to \mathbb{G}M \to 0\]
	is exact. Applying \(\Hom_{\Gamma}(\mathbb{G}(\upnu_{i}M),-)\)  and dropping Hom's gives a commutative diagram
	\[{\scriptsize
		\begin{array}{c}
			\begin{tikzpicture}
				\node (a1) at (-2,0) {$ 0$};
				\node (a2) at (0,0) {${}_{\Gamma}(\mathbb{G}(\upnu_{i}M), \mathbb{G}(\upnu_{i}M))$};
				\node (a3) at (3,0) {${}_{\Gamma}(\mathbb{G}(\upnu_{i}M), \mathbb{G}V) $};
				\node (a4) at (6,0) {${}_{\Gamma}(\mathbb{G}(\upnu_{i}M), \mathbb{G}M)$};
				\node (a5) at (9,0) {$\Ext^{1}_{\Gamma}(\mathbb{G}(\upnu_{i}M), \mathbb{G}(\upnu_{i}M))$};
				\node (b1) at (-2,-1.5) {$ 0 $};
				\node (b2) at (0,-1.5) {$ {}_{\scrR}(\upnu_{i}M, \upnu_{i}M)$};
				\node (b3) at (3,-1.5) {$ {}_{\scrR}(\upnu_{i}M, V)$};
				\node (b4) at (6,-1.5) {$ {}_{\scrR}(\upnu_{i}M, M)$};
				\node (b5) at (9,-1.5) {$ 0 $};
				\draw[->] (a1) to node[above] {$ $} (a2);
				\draw[->] (a2) to node[above] {$ $} (a3);
				\draw[->] (a3) to node[above] {$ $} (a4);
				\draw[->] (a4) to node[above] {$ $} (a5);
				\draw[->] (b1) to node[above] {$ $} (b2);
				\draw[->] (b2) to node[left] {$ $} (b3);
				\draw[->] (b3) to node[right] {$ $} (b4);
				\draw[->] (b4) to node[right] {$ $} (b5);
				\draw [->](a2) to node[left] {$\cong $} node[right] {$\mathrm{ref~equiv}$} (b2);
				\draw [->](a3) to node[left] {$\cong $} node[right] {$\mathrm{ref~equiv}$} (b3);
				\draw [->](a4) to node[left] {$\cong $} node[right] {$\mathrm{ref~equiv}$} (b4);
			\end{tikzpicture}
	\end{array}}
	\] in which the top line is exact, where the vertical arrows are reflexive equivalences by \cite{IR}.
	
	By e.g., \cite[6.10(2)]{IW2}, \(\End_{\Gamma}(\mathbb{G}(\upnu_{i}M)) \cong \End_{\scrR}(\upnu_{i}M)\in \mathrm{CM}~\scrR, \) so \[\fl_{\scrR}\Ext^{1}_{\Gamma}(\mathbb{G}(\upnu_{i}M), \mathbb{G}(\upnu_{i}M))=0,\] by \cite[2.7]{IW2}. Since \(\scrR\)  is isolated, \(\Ext^{1}=0.\) Thus the bottom line is also exact, as claimed.
	
	Since \eqref{exact sequence} is exact, applying \(\Hom_{\scrR}(M, \upnu_{i}M)\otimes_{   \upnu_{\kern-1pt i}\kern-2pt \Bnohat}-\) gives an exact sequence  in the top line of the following commutative diagram {\scriptsize
		\[
		\begin{array}{c}
			\begin{tikzpicture}[trim left]
				\node (a2) at (1.9,0) {${}_{\scrR}(M,\upnu_{i}M)\otimes_{  \upnu_{\kern-1pt i}\kern-2pt \Bnohat}{}_{\scrR}(\upnu_{i}M,\upnu_{i}M)$};
				\node (a3) at (6.4,0) {${}_{\scrR}(M,\upnu_{i}M)\otimes_{ \upnu_{\kern-1pt i}\kern-2pt \Bnohat}{}_{\scrR}(\upnu_{i}M, V) $};
				\node (a4) at (10.9,0) {${}_{\scrR}(M,\upnu_{i}M)\otimes_{ \upnu_{\kern-1pt i}\kern-2pt \Bnohat}{}_{\scrR}(\upnu_{i}M, M)$};
				\node (a5) at (13.4,0) {$0 $};
				\node (c1) at (10.9,-2) {$\Bnohat$};
				\node (b2) at (1.9,-1.5) {$ {}_{\scrR}(M, \upnu_{i}M)$};
				\node (b3) at (6.4,-1.5) {$ {}_{\scrR}(M, V)$};
				\node (b4) at (10.9,-1.5) {$ {}_{\scrR}(M, M)$};
				\node (b5) at (12.4,-1.5) {$ \frac{\Bnohat}{\Bnohat e_{i} \Bnohat}$};
				\node (b6) at (13.4,-1.5) {$0,$};
				\draw[->] (a2) to node[above] {$ $} (a3);
				\draw[->] (a3) to node[above] {$ $} (a4);
				\draw[->] (a4) to node[above] {$ $} (a5);
				\draw[->] (b2) to node[left] {$ $} (b3);
				\draw[->] (b3) to node[right] {$ $} (b4);
				\draw[->] (b4) to node[right] {$ $} (b5);
				\draw[->] (b5) to node[right] {$ $} (b6);
				\draw [->](a2) to node[right] {$ $} (b2);
				\draw [->](a3) to node[right] {$ $}  (b3);
				\draw [->](a4) to node[right] {$ $} (b4);
				\path (b4)--(c1) node[midway]{$\scriptstyle\parallel$};
			\end{tikzpicture}
		\end{array}
		\]}whilst applying \(\Hom_{\scrR}(M,-)\) to \eqref{exchange sequence} gives the bottom line (see \cite[A.7(3)]{W1}). The leftmost vertical maps are isomorphisms since \(\Hom_{\scrR}(\upnu_{i}M, \upnu_{i}M)\) and \(\Hom_{\scrR}(\upnu_{i}M, V)\) are projective \(   \upnu_{\kern-1pt i}\kern-2pt \Bnohat\)-modules. 
	Thus 
	\[ \Hom_{\scrR}(M,\upnu_{i}M)\otimes_{   \upnu_{\kern-1pt i}\kern-2pt \Bnohat}\Hom_{\scrR}(\upnu_{i}M, M)\cong \ker(\Bnohat \to \tfrac{\Bnohat}{\Bnohat e_{i} \Bnohat})= \Bnohat e_{i} \Bnohat=I_{i}\]  via the map
	\[ \Hom_{\scrR}(M,\upnu_{i}M)\otimes_{   \upnu_{\kern-1pt i}\kern-2pt \Bnohat}\Hom_{\scrR}(\upnu_{i}M,M) \to  I_{i}\] sending \(a\otimes b\mapsto b\circ a.\)
	Since this is  clearly a bimodule map, thus  bimodule isomorphism, the statement follows.
\end{proof}
\begin{lemma}
	There is a functorial isomorphism 
	\[
	\Upphi_{\upalpha}\circ \Upphi_{i}^{2} \circ \Upphi_{\upalpha}^{-1}\cong \RHom_{\Anohat}(\uptau_{\upalpha} \otimes^{\bf{L}}_{\Bphat} I_{i}\otimes^{\bf{L}}_{\Bphat} \uptau_{\upalpha}^{*},-).
	\]
\end{lemma}
\begin{proof}
	For  the \(\Bnohat\)-bimodule \(I_{i}\) defined by the natural exact sequence \(I_{i}\to \Bnohat\to  \Bnohat_{i},\) by Claim \ref{infinite Ii}  \(\Upphi_{i} \circ \Upphi_{i}\cong \RHom_{\Bnohat}(I_{i},-).\) So,
	\begin{align*}
		\Upphi_{\upalpha}\circ \Upphi_{i}^{2} \circ \Upphi_{\upalpha}^{-1} &= \RHom_{\Anohat}(\uptau_{\upalpha},-) \circ \RHom_{\Bnohat}(I_{i},-)\circ \RHom_{\Anohat}(\uptau_{\upalpha}^{*},-)\\
		&\cong \RHom_{\Anohat}(\uptau_{\upalpha} \otimes^{\bf{L}}_{\Bphat} I_{i}\otimes^{\bf{L}}_{\Bphat} \uptau _{\upalpha}^{*},-),
	\end{align*}
	as required.
\end{proof}
The point of the following result is that we can view the functor \(\TW_{\upalpha,i}\) as a lift of the complete local equivalence \(\Upphi_{\upalpha}\circ \Upphi_{i}^{2} \circ \Upphi_{\upalpha}^{-1}.\)
\begin{prop}\label{02 0f 06 march 2023} Under the Zariski local setup \eqref{01 of 23 april 2023}, the following hold.
	\begin{enumerate}
		\item\label{02 0f 06 march 2023 one} There exists a triangle in the derived category of $\hat{\Lambda}$-$\hat{\Lambda}$  bimodules
		\begin{equation}\label{02 0f 24 jan 2023}  
			\mathbb{M}\otimes^{\bf{L}}_{\Aphat}  \uptau_{\upalpha}\otimes^{\bf{L}}_{\Bphat} I_{i}\otimes^{\bf{L}}_{\Bphat} \uptau^{*}_{\upalpha} \otimes^{\bf{L}}_{\Aphat} \mathbb{M}^{*} \to {}_{\hat{\Lambda}}\hat{\Lambda}_{\hat{\Lambda}}\xrightarrow{h} \mathbb{M}\otimes^{\bf{L}}_{\Aphat}  \uptau_{\upalpha}\otimes^{\bf{L}}_{\Bphat} Z_{\upalpha,i}. 
		\end{equation} 
		\item\label{02 0f 06 march 2023 two}  There is a commutative diagram
		\[
		\begin{array}{c}
			\begin{tikzpicture}
				\node (a1) at (0,0) {$\D(\hat{\Lambda})$};
				\node (a2) at (2,0) {$\D(\Anohat) $};
				\node (a3) at (4,0) {$\D(\Bnohat)$};
				\node (a4) at (6,0) {$\D(\Bnohat)$};
				\node (a5) at (8,0) {$\D(\Anohat)$};
				\node (a6) at (10,0) {$\D(\hat{\Lambda})$};
				\node (b1) at (0,-1.5) {$\D(\Lambda)$};
				\node (b2) at (10,-1.5) {$\D(\Lambda)$};
				\draw[->] (a1) to node[above] {$\scriptstyle  \mathbb{F}^{-1}$} (a2);
				\draw[->] (a2) to node[above] {$\scriptstyle \Upphi_{\upalpha}^{-1}$} (a3);
				\draw[->] (a3) to node[above] {$\scriptstyle {\Upphi}_{i}^{2}$} (a4);
				\draw[->] (a4) to node[above] {$\scriptstyle \Upphi_{\upalpha}$} (a5);
				\draw[->] (a5) to node[above] {$\scriptstyle \mathbb{F}$} (a6);
				\draw[->] (b1) to node[above] {$\scriptstyle\TW_{\upalpha,i}$} (b2);
				\draw[<-] (b1) to node[left] {$\scriptstyle \mathrm{F}$} (a1);
				\draw[<-] (b2) to node[right] {$\scriptstyle  \mathrm{F} $} (a6);
			\end{tikzpicture}			
		\end{array}
		\]
	\end{enumerate}
\end{prop}
\begin{proof}
	(1) The result is implied by tensoring the exact sequence \(I_{i}\to \Bnohat\to  \Bnohat_{i},\) using the isomorphisms \(\uptau_{\upalpha} \otimes \uptau_{\upalpha}^{*} \cong {}_{\Anohat}\!\Anohat_{\Anohat} \) and \(\mathbb{M} \otimes \mathbb{M}^{*}\cong {}_{\hat{\Lambda}}\hat{\Lambda}_{\hat{\Lambda}}.\)\\
	(2) Observe that  
	\begin{align*}
		\TW_{\upalpha,i} \circ\, \mathrm{F}&= \RHom_{\hat{\Lambda}}({}_{\Lambdaphat}C\otimes_{\Lambdaphat}^{\bf{L}}\hat{\Lambda},-),\\ 
		\mathrm{F}\circ (\mathbb{F}\circ \Upphi_{\upalpha }\circ \Upphi^{2}_{i}\circ \Upphi_{\upalpha }^{-1}\circ \mathbb{F}^{-1})&=\RHom_{\hat{\Lambda}}({}_{\Lambdaphat}\hat{\Lambda}\otimes_{\hat{\Lambda}}^{\bf{L}}\mathbb{M}\otimes^{\bf{L}}_{\Aphat}  \uptau_{\upalpha}\otimes^{\bf{L}}_{\Bphat} I_{i}\otimes^{\bf{L}}_{\Bphat} \uptau^{*}_{\upalpha} \otimes^{\bf{L}}_{\Aphat} \mathbb{M}^{*},-),
	\end{align*} 
	thus it suffices to prove that \(C\otimes \hat{\Lambda}\cong \hat{\Lambda}\otimes \mathbb{M}\otimes \uptau_{\upalpha}\otimes I_{i}\otimes\uptau_{\upalpha}^{*}\otimes \mathbb{M}^{*} \) in the derived category of  \(\Lambda\)-\(\hat{\Lambda}\) bimodules, where to ease notation, we have written \(\otimes \) instead of \(\otimes^{\bf{L}}_{}.\)

	Tensoring   \eqref{01 of 24 january 2023} by \(\otimes^{\bf{L}}_{\Lambda}\hat{\Lambda}_{\hat{\Lambda}}\) gives  \begin{equation}\label{03 0f 01 jan 2023} 
		{}_{\Lambdaphat}C\otimes^{\bf{L}}_{\Lambdaphat} \hat{\Lambda}_{\hat{\Lambda}} \to  {}_{\Lambdaphat}\Lambda\otimes^{\bf{L}}_{\Lambdaphat}\hat{\Lambda}_{\hat{\Lambda}} 
		\xrightarrow{f\otimes \Id} 
		{}^{\phantom{L}}_{\Lambdaphat}\hat{\Lambda}\otimes^{\bf{L}}_{\hat{\Lambda}}\mathbb{M}\otimes^{\bf{L}}_{\Aphat}  \uptau_{\upalpha}\otimes^{\bf{L}}_{\Bphat} Z_{\upalpha,i}\otimes^{\bf{L}}_{\hat{\Lambda}} \hat{\Lambda}\otimes^{\bf{L}}_{\Lambdaphat} \hat{\Lambda}_{\hat{\Lambda}} \to 
	\end{equation} 
	Since \(\upalpha\) is an atom, by Lemma \ref{02 of 24 january 2023} \( \Upphi_{\upalpha}^{-1}(\Bnohat_{i})\) is in one degree only, and by the proof  it has finite length. Since \(\mathbb{F}\) is  induced from a Morita equivalence, \(\Upphi^{-1}\mathbb{F}^{-1}(\Bnohat_{i})\) is also one degree only, and has  finite length.
	
	Now,  \( \Upphi^{-1}_{\upalpha}\mathbb{F}^{-1}(\Bnohat_{i})=\Bnohat_{i}\otimes^{\bf{L}}_{\Bphat}\, \uptau_{\upalpha}^{*}\otimes^{\bf{L}}_{\Aphat} \mathbb{M}^{*}\) as one sided modules, which implies that \(\Bnohat_{i}\otimes^{\bf{L}}_{\Bphat}\, \uptau_{\upalpha}^{*}\otimes^{\bf{L}}_{\Aphat} \mathbb{M}^{*}\) is one degree only and in that degree the cohomology is finite dimensional. 
	
	Now,  the natural  map \(\upbeta \colon Z_{\upalpha,i}\otimes_{\hat{\Lambda}}(\hat{\Lambda}\otimes_{\Lambdaphat} \hat{\Lambda}) \to {}_{\Bphat}{Z_{\upalpha,i}}_{\hat{\Lambda}}\) sending \(z\otimes r\otimes s\mapsto zrs\) is a bijection by Lemma \ref{04 0f 24 january 2023}\eqref{04 0f 24 january 2023 two}, so  since  this is clearly a bimodule homomorphism, it follows that  \(Z_{\upalpha,i}\otimes \hat{\Lambda} \otimes  \hat{\Lambda} \cong  Z_{\upalpha,i}\) as \(\Bnohat\)-$\hat{\Lambda}$ bimodules.
	
	Now, the following diagram commutes
	\begin{equation}\label{05 0f 24 january 2023}
		\begin{array}{c}
			\begin{tikzpicture}
				\node (a2) at (5,0) {${}_{\Lambdaphat}\Lambda \otimes_{\Lambdaphat} \hat{\Lambda}_{\hat{\Lambda}}$};
				\node (a3) at (10,0) {${}_{\Lambdaphat}\hat{\Lambda}\otimes\mathbb{M}\otimes \uptau_{\upalpha} \otimes Z_{\upalpha,i}\otimes \hat{\Lambda}\otimes  \hat{\Lambda}_{\hat{\Lambda}}$};
				\node(c1) at (5,-1){${}_{\Lambdaphat}\hat{\Lambda}_{\hat{\Lambda}}$};
				\node (b2) at (5,-2) {${}_{\Lambdaphat}\hat{\Lambda} \otimes_{\hat{\Lambda}} \hat{\Lambda}_{\hat{\Lambda}}$};
				\node (b3) at (10,-2) {${}_{\Lambdaphat}\hat{\Lambda}\otimes \mathbb{M}\otimes  \uptau_{\upalpha}\otimes {Z_{ \upalpha,i}}_{\hat{\Lambda}}$};
				\draw[->] (a2) to node[above] {$\scriptstyle f\otimes \Id$} (a3);
				\draw[->] (a3) to node[right] {$\scriptstyle\Id \otimes \uppsi$} (b3);
				\draw[->] (b2) to node[above] {$\scriptstyle\Id\otimes h$} (b3);
				\path (a2)--(c1) node[midway]{$\scriptstyle\parallel$};
				\path (c1)--(b2) node[midway]{$\scriptstyle\parallel$};
			\end{tikzpicture}
		\end{array}
	\end{equation}
	since
	\[
	\begin{array}{c}
		\begin{tikzpicture}
			\node (a1) at (0,0) {${}_{\Lambdaphat}\Lambda \otimes_{\Lambda} \hat{\Lambda}_{\hat{\Lambda}}$};
			\node (a2) at (5,0) {${}_{\Lambdaphat}\hat{\Lambda}\otimes\mathbb{M}\otimes \uptau_{\upalpha} \otimes Z \otimes \hat{\Lambda}\otimes  \hat{\Lambda}_{\hat{\Lambda}}$};
			\node(c1) at (0,-1){${}_{\Lambdaphat}\hat{\Lambda}_{\hat{\Lambda}}$};
			\node (b1) at (0,-2) {${}_{\Lambdaphat}\hat{\Lambda} \otimes_{\hat{\Lambda}} \hat{\Lambda}_{\hat{\Lambda}}$};
			\node (b2) at (5,-2) {${}_{\Lambdaphat}\hat{\Lambda}\otimes \mathbb{M}\otimes  \uptau_{\upalpha}\otimes Z$};
			\node(a3) at (-1.3,0){$\scriptstyle 1\otimes r$};
			\node(c2) at (-1.3,-1){$\scriptstyle r$};
			\node(b3) at (-1.3,-2){$\scriptstyle r \otimes 1$};
			\node(d1) at (-1.3,-2.3){};
			\node(d2) at (5,-2.5){$\scriptstyle r\otimes h(1)$};
			\node(e1) at (-1.3,0.3){};
			\node(e2) at (5,0.5){$\scriptstyle 1\otimes h(1)\otimes 1\otimes r$};
			\draw[->] (a1) to node[below] {$\scriptstyle f\otimes \Id$} (a2);
			\draw[->] (a2) to node[right] {$\scriptstyle\Id \otimes \uppsi$} (b2);
			\draw[->] (b1) to node[above] {$\scriptstyle\Id\otimes h$} (b2);
			\draw[->] (c1) to node[above] {} (a1);
			\draw[->] (c1) to node[above] {} (b1);
			\draw [| ->] (c2) edge (a3);
			\draw [| ->] (c2) edge (b3);
			\draw [line width = 1 pt,| ->, bend right=10] (d1) edge (d2);
			\draw [line width = 1 pt,| ->, bend left=10] (e1) edge (e2);
		\end{tikzpicture}
	\end{array}
	\]
	and  \(1\otimes h(1)r=1\otimes h(r)=1\otimes rh(1)=r\otimes h(1).\)  Completing the top and bottom lines in \eqref{05 0f 24 january 2023} to triangles, using \eqref{02 0f 24 jan 2023} 
	\begin{equation*}\label{07 0f 24 january 2023}{\scriptsize
			\begin{array}{c}
				\begin{tikzpicture}
					\node (a1) at (0,0) {${}_{\Lambdaphat}C\otimes \hat{\Lambda}_{\hat{\Lambda}}$};
					\node (a2) at (4,0) {${}_{\Lambdaphat}\Lambda \otimes_{\Lambda} \hat{\Lambda}_{\hat{\Lambda}}$};
					\node (a3) at (9,0) {${}_{\Lambdaphat}\hat{\Lambda}\otimes\mathbb{M}\otimes \uptau_{\upalpha} \otimes \Bnohat_{i}\otimes \uptau^{*}_{\upalpha} \otimes \mathbb{M}^{*}\otimes \hat{\Lambda}\otimes  \hat{\Lambda}_{\hat{\Lambda}}$};
					\node(c1) at (4,-0.75){${}_{\Lambdaphat}\hat{\Lambda}_{\hat{\Lambda}}$};
					\node (b1) at (0,-1.5) {${}_{\Lambdaphat}\hat{\Lambda} \otimes \mathbb{M}\otimes \uptau_{\upalpha}\otimes I_{i}\otimes \uptau^{*}_{\upalpha} \otimes\mathbb{M}^{*}_{\hat{\Lambda}}  $};
					\node (b2) at (4,-1.5) {${}_{\Lambdaphat}\hat{\Lambda} \otimes_{\hat{\Lambda}} \hat{\Lambda}_{\hat{\Lambda}}$};
					\node (b3) at (9,-1.5) {${}_{\Lambdaphat}\hat{\Lambda}\otimes \mathbb{M}\otimes  \uptau_{\upalpha}\otimes \Bnohat_{i}\otimes \uptau^{*}_{\upalpha} \otimes \mathbb{M}^{*}_{\hat{\Lambda}}$};
					\draw[->] (a1) to node[above] {} (a2);
					\draw[->] (a2) to node[above] {$f\otimes \Id$} (a3);
					\draw[->] (a3) to node[right] {$\Id \otimes \uppsi$} (b3);
					\draw [ dashed,->] (a1) edge node[right]{$\exists$} (b1);
					\draw[->] (b1) to node[above] {} (b2);
					\draw[->] (b2) to node[above] {$\Id\otimes h$} (b3);
					\path (a2)--(c1) node[midway]{$\parallel$};
					\path (c1)--(b2) node[midway]{$\parallel$};
				\end{tikzpicture}
		\end{array}}
	\end{equation*}
	Since the rightmost vertical maps are isomorphisms, it follows that so is the leftmost, and thus  \({}_{\Lambdaphat}C\otimes \hat{\Lambda}_{\hat{\Lambda}} \cong {}_{\Lambdaphat}\hat{\Lambda} \otimes \mathbb{M}\otimes \uptau_{\upalpha}\otimes I_{i}\otimes \uptau^{*}_{\upalpha} \otimes\mathbb{M}^{*}_{\hat{\Lambda}} \), as required.
\end{proof}

\begin{cor}\label{01 0f 02 march 2023}
	The following diagram commutes
	\[{\scriptsize
		\begin{array}{c}
			\begin{tikzpicture}
				\node (a1) at (0,0) {$\D(\hat{\Lambda})$};
				\node (a2) at (2,0) {$\D(\Anohat) $};
				\node (a3) at (4,0) {$\D(\Bnohat)$};
				\node (a4) at (6,0) {$\D(\Bnohat)$};
				\node (a5) at (8,0) {$\D(\Anohat)$};
				\node (a6) at (10,0) {$\D(\hat{\Lambda})$};
				\node (b1) at (0,-1.5) {$\D(\Lambda)$};
				\node (b2) at (10,-1.5) {$\D(\Lambda)$};
				\draw[->] (a1) to node[above] {$\scriptstyle  \mathbb{F}^{-1}$} (a2);
				\draw[->] (a2) to node[above] {$\scriptstyle \Upphi_{\upalpha}^{-1}$} (a3);
				\draw[->] (a3) to node[above] {$\scriptstyle {\Upphi}_{i}^{-2}$} (a4);
				\draw[->] (a4) to node[above] {$\scriptstyle \Upphi_{\upalpha}$} (a5);
				\draw[->] (a5) to node[above] {$\scriptstyle \mathbb{F}$} (a6);
				\draw[->] (b1) to node[above] {$\TW_{\upalpha,i}^{*}$} (b2);
				\draw[->] (b1) to node[left] {$\scriptstyle \mathrm{F}^{\LA}$} (a1);
				\draw[->] (b2) to node[right] {$\scriptstyle \mathrm{F}^{\LA} $} (a6);
			\end{tikzpicture}			
	\end{array}}
	\] 
\end{cor}
\begin{proof}
	Take \(\LA\) of \ref{02 0f 06 march 2023}\eqref{02 0f 06 march 2023 two}, using  Remark \ref{01 of 21 feb 2023} to see that \(\TW_{\upalpha,i}^{*}\) is left adjoint to \(\TW_{\upalpha,i}.\)
\end{proof}
To set notation, given \((\upalpha, i)\) in the finite arrangement \(\scrH,\) say \(\upalpha \colon C_{+} \to D.\) Write \(\hat{\Lambda}_{\con}\) for the contraction algebra in chamber \(C_{+}\) and  \(\hat{\Gamma}_{\con}\) for the contraction algebra in chamber \(D.\) Mimicking the introduction and Proposition \ref{02 0f 06 march 2023}, consider the composition 
\[ {\scriptsize \mathrm{J}_{\upalpha,i}\colonequals \Db(\hat{\Lambda}_{\con})\xrightarrow[\sim]{F_{\upalpha}} \Db(\hat{\Gamma}_{\con}) \xrightarrow[\sim]{F_{i}^{2}} \Db(\hat{\Gamma}_{\con})   \xrightarrow[\sim]{F_{\upalpha}^{-1}} \Db(\hat{\Lambda}_{\con})}, \]   where \(F_{i}\) and \(F_{\upalpha}\) are the standard equivalences in \cite{A}.

The following lifts results from \cite{A} to the Zariski local setting.
\begin{cor}\label{jenny and caro diag }
	For any choice of \((\upalpha,i)\) in the finite arrangement \(\scrH, \) the following diagram commutes
	\[
	\begin{tikzpicture}
		\node (a1) at (0,0) {$\D(\hat{\Lambda}_{\mathrm{con}})$};
		\node (a2) at (3,0) {$\D(\hat{\Lambda}_{\mathrm{con}})$};
		\node (c1) at (0,-1.5) {$\D(\Lambda)$};
		\node (c2) at (3,-1.5) {$\D(\Lambda)$};
		\draw[->] (a1) to node[above] {$\scriptstyle  \mathrm{J}_{\upalpha,i}$} (a2);
		\draw[->] (c1) to node[above] {$\scriptstyle \TW_{\upalpha,i}$} (c2);
		\draw[->] (a1) to node[right] {$ $} (c1);
		\draw[->] (a2) to node[left] {$ $} (c2);							
	\end{tikzpicture}
	\] 
\end{cor}
\begin{proof}
	The result follows from the following
	\[
	\begin{array}{c}
		\begin{tikzpicture}
			\node (d1) at (2,1.5) {$\D(\hat{\Lambda}_{\con})$};
			\node (d2) at (8,1.5) {$\D(\hat{\Lambda}_{\con})$};
			\node (a1) at (0,0) {$\D(\hat{\Lambda})$};
			\node (a2) at (2,0) {$\D(\Anohat) $};
			\node (a3) at (4,0) {$\D(\Bnohat)$};
			\node (a4) at (6,0) {$\D(\Bnohat)$};
			\node (a5) at (8,0) {$\D(\Anohat)$};
			\node (a6) at (10,0) {$\D(\hat{\Lambda})$};
			\node (b1) at (0,-1.5) {$\D(\Lambda)$};
			\node (b2) at (10,-1.5) {$\D(\Lambda)$};
			\draw[->] (a1) to node[above] {$\scriptstyle  \mathbb{F}^{-1}$} (a2);
			\draw[->] (a2) to node[above] {$\scriptstyle \Upphi_{\upalpha}^{-1}$} (a3);
			\draw[->] (a3) to node[above] {$\scriptstyle {\Upphi}_{i}^{2}$} (a4);
			\draw[->] (a4) to node[above] {$\scriptstyle \Upphi_{\upalpha}$} (a5);
			\draw[->] (a5) to node[above] {$\scriptstyle \mathbb{F}$} (a6);
			\draw[->] (b1) to node[above] {$\scriptstyle\TW_{\upalpha,i}$} (b2);
			\draw[<-] (b1) to node[left] {$\scriptstyle \mathrm{F}$} (a1);
			\draw[<-] (b2) to node[right] {$\scriptstyle  \mathrm{F} $} (a6);
			\draw[->] (d1) to node[above] {$\scriptstyle  \mathrm{J}_{\upalpha,i}$} (d2);
			\draw[->] (d1) to node  {$ $} (a2);
			\draw[->] (d2) to node[left] {$ $} (a5);	
			\draw[rounded corners,->] (1.7,1.2) -- (1.7,0.4) -- (-0.5,0.4) -- (-0.5,-1.2);
			\draw[rounded corners,->] (8.3,1.2) -- (8.3,0.4) -- (10.5,0.4) -- (10.5,-1.2);
		\end{tikzpicture}			
	\end{array}
	\]
	The  top diagram commutes by \cite{A} and the bottom diagram is  Proposition \ref{02 0f 06 march 2023}\eqref{02 0f 06 march 2023 two}.
\end{proof}
\section{One and Two sided tilting}
In this section, we will prove  that for any $(\upalpha, i),$ the functor \(\TW_{\upalpha,i} \colon \Db(\Lambda)\to \Db(\Lambda)\) is an equivalence.
\subsection{Tilting Generalities}  
Throughout this subsection, \(S\) is a commutative Noetherian ring, and \(\Gamma\) is a module-finite \(S\)-algebra.
\begin{lemma}\label{03 0f 06 march 2023}
	If \( M\in \Mod S\), then \(M=0\) if and only if \(M_{\mathfrak{p}} \otimes_{S_{\mathfrak{p}}}\hat{S}_{\mathfrak{p}}=0\) for all \(\mathfrak{p} \in \Spec S.\)
\end{lemma}
\begin{proof}
	By \cite[Proposition 3.8 ]{AM} \(M=0\) iff \(M_{\mathfrak{p}}=0\) for all \(\mathfrak{p} \in \Spec S.\) By faithful flatness \cite[Theorem 7.2]{M}, this is  equivalent to  \( M_{\mathfrak{p}} \otimes_{S_{\mathfrak{p}}}\hat{S}_{\mathfrak{p}}=0\) for all \(\mathfrak{p} \in \Spec S.\)
\end{proof}

\begin{lemma}\label{01 of 13 march 2023}
	If  \(T\in\Kb  (\smallproj \Gamma),\) then there are isomorphisms
	\begin{align*}
		\Hom_{\D(\Gamma)}(T_{\Gamma}, T_{\Gamma})\otimes_{S} S_{\mathfrak{p}} &\cong \Hom_{\D(\Gamma_{\mathfrak{p}})}(T_{\mathfrak{p}}, T_{\mathfrak{p}}) \\
		\Hom_{\D(\Gamma_{\mathfrak{p}})}(T_{\mathfrak{p}}, T_{\mathfrak{p}})\otimes_{S_{\mathfrak{p}}} \hat{S}_{\mathfrak{p}} &\cong  \Hom_{\D(\hat{\Gamma})}(T_{\mathfrak{p}} \otimes_{S_{\mathfrak{p}}} \hat{S}_{\mathfrak{p}},   T_{\mathfrak{p}}\otimes_{S_{\mathfrak{p}}} \hat{S}_{\mathfrak{p}}) 
	\end{align*}
\end{lemma}	
\begin{proof}
	This is Iyama-Reiten, proof of (3) and (4) in \cite[Theorem 3.1]{IR}.
\end{proof}

\begin{definition}\label{ 01 0f 21 march 2023}
	A \emph{one-sided tilting complex} for \(\Gamma\) is an object \(T\in \Kb (\smallproj \Gamma)\) such that 
	\begin{enumerate}
		\item \(\Hom(T, T[t])=0\) for all \(t \neq 0\)
		\item If \(x\in \D(\Mod \Gamma)\) with \(\RHom_{\Gamma}(T,x)=0\), then \(x\cong 0.\)
	\end{enumerate}
\end{definition}
By \cite[p.80]{K} Definition \ref{ 01 0f 21 march 2023} is equivalent to the other definitions in the literature.

\begin{lemma}\label{01 of 16 march 2023}
	Given \(T\in \Kb(\smallproj \Gamma),\) the following are equivalent.
	\begin{enumerate}
		\item \(T\) is one-sided  tilting complex
		\item\(T\otimes_{S} S_{\mathfrak{p}}\) is one-sided tilting \(\Gamma_{\mathfrak{p}}\)-complex for all \(\mathfrak{p} \in \Spec S\)
		\item  \(T\otimes_{S} S_{\mathfrak{p}} \otimes \hat{S}_{\mathfrak{p}}\) is one-sided tilting \(\hat{\Gamma}_{\mathfrak{p}}\)-complex for all \(\mathfrak{p} \in \Spec S\)
	\end{enumerate}
\end{lemma}
\begin{proof}
	(1)\(\Rightarrow\)(2). This holds by e.g  \cite[Lemma 2.7]{IR}.\\
	(2)\(\Rightarrow\)(3). We show \(T_{\mathfrak{p}}\) is tilting implies \(A=T_{\mathfrak{p}}\otimes_{S_{\mathfrak{p}}} \hat{S}_{\mathfrak{p}}\) is tilting. It is clear, since \(\otimes_{S_{\mathfrak{p}}}\hat{S}_{\mathfrak{p}}\) is exact, taking projectives to projectives, that \(A \in \Kb(\smallproj \hat{\Gamma}_{\mathfrak{p}}).\)
	
	We know that   \( \Hom_{\D(\Gamma_{p})}(T_{\mathfrak{p}}, T_{\mathfrak{p}}[t])=0\) for all \(t \neq 0\) by assumption. Thus by  Lemma \ref{01 of 13 march 2023}, for all \(t \neq 0\) we have \[0=\Hom_{\D(\Gamma_{\mathfrak{p}})}(T_{\mathfrak{p}}, T_{\mathfrak{p}}[t])\otimes_{S_{\mathfrak{p}}} \hat{S}_{\mathfrak{p}} \cong  \Hom_{\D(\hat{\Gamma})}(T_{\mathfrak{p}} \otimes_{S_{\mathfrak{p}}} \hat{S}_{\mathfrak{p}},   T_{\mathfrak{p}}[t]\otimes_{S_{\mathfrak{p}}} \hat{S}_{\mathfrak{p}})=\Hom_{\D(\hat{\Gamma})}(A,   A[t]).\] 
	
	Now let \(x\in \D(\Mod \hat{\Gamma}_{\mathfrak{p}})\) such that \(\RHom_{\hat{\Gamma}_{\mathfrak{p}}}(A,x)=0.\) Then for all \(t\in \mathbb{Z}\)
	\begin{align*}
		0=\mathrm{H}^{t}(\RHom_{\hat{\Gamma}_{\mathfrak{p}}}(A,x))&=\Hom_{\D(\hat{\Gamma}_{\mathfrak{p}})}(A, x[t])\\
		&=\Hom_{\D(\hat{\Gamma}_{\mathfrak{p}})}(T_{\mathfrak{p}}\otimes_{S_{\mathfrak{p}}} \hat{S}_{\mathfrak{p}}, x[t])\\
		&=\Hom_{\D(\Gamma_{\mathfrak{p}})}(T_{\mathfrak{p}}, \res(x)[t]).
	\end{align*}  Thus \( \RHom_{\Gamma_{\mathfrak{p}}}(T_{\mathfrak{p}}, \res(x))=0,\) so since \(T_{\mathfrak{p}}\) is tilting, \(\res(x)\cong 0.\) Since restriction of scalars is exact, \(x \cong 0\)\\
	(3)\(\Rightarrow\)(1). By assumption \[0=\Hom_{\D(\hat{\Gamma}_{\mathfrak{p}})}(T_{\mathfrak{p}} \otimes_{S_{\mathfrak{p}}} \hat{S}_{\mathfrak{p}},   T_{\mathfrak{p}}[t]\otimes_{S_{\mathfrak{p}}} \hat{S}_{\mathfrak{p}}) \cong \Hom_{\D(\Gamma)}(T,   T[t] ) \otimes_{S_{\mathfrak{p}}}  \hat{S}_{\mathfrak{p}}.\] Thus   \(\Hom(T, T[t])=0\) for all \(t \neq 0\), since \(-\otimes_{S_{\mathfrak{p}}} \hat{S}_{\mathfrak{p}}\) is faithfully flat  \cite[Theorem 7.2]{M}.

	By assumption \(T\in \Kb( \smallproj \Gamma).\) Let \(x \in \D(\Mod \Gamma )\)  such that \(\RHom (T,x)=0.\) Then for all \(t\in \mathbb{Z}, \Hom_{\D(\Gamma)}(T, x[t])=0.\) Thus, tensoring by \( {S_{\mathfrak{p}}},\) using Lemma \ref{01 of 13 march 2023} \[ \Hom_{\D(\Gamma)}(T, x[t]) \otimes_{S} S_{\mathfrak{p}}  \otimes_{S_{\mathfrak{p}}} \hat{S}_{\mathfrak{p}} \cong \Hom_{\D(\hat{\Gamma}_{\mathfrak{p}})}(\hat{T}_{\mathfrak{p}}, \hat{x}_{\mathfrak{p}}[t]).\] Since \(\hat{T}_{\mathfrak{p}}\) is tilting, \(\hat{x}_{\mathfrak{p}} \cong 0.\) Thus, since \(-\otimes S_{\mathfrak{p}} \) and  \(-\otimes \hat{S}_{\mathfrak{p}} \) are exact, for all \(t\in \mathbb{Z}\) we have \[ \mathrm{H}^{t}(x)\otimes_{S_{\mathfrak{p}}} S_{\mathfrak{p}} \otimes_{S_{\mathfrak{p}}} \hat{S}_{\mathfrak{p}} \cong \mathrm{H}^{t}(x\otimes_{S_{\mathfrak{p}}} S_{\mathfrak{p}} \otimes_{S_{\mathfrak{p}}} \hat{S}_{\mathfrak{p}}) =0. \] But \(\otimes_{S_{\mathfrak{p}}} \hat{S}_{\mathfrak{p}}\) is faithfully flat, so by  Lemma \ref{03 0f 06 march 2023} \(\mathrm{H}^{t}(x)\cong 0\) for all \(t\in \mathbb{Z}, \) so \(x\cong 0\) as claimed.
\end{proof}

\begin{cor}\label{02 of 16 march 2023}
	Suppose \(T\in \Kb(\smallproj \Gamma)\) and choose \(\mathfrak{m} \in \Max S\). If \(T_{\mathfrak{p}}\) is tilting for all \(\mathfrak{p} \neq \mathfrak{m}\) and \(T_{\mathfrak{m}}\otimes \hat{S}_{\mathfrak{m}}\) is tilting, then \(T\) is tilting.
\end{cor}
\begin{proof}
	The proof follows from Lemma \ref{01 of 16 march 2023} by considering the two cases, one when  the prime ideal \(\mathfrak{p}\) is not equal to the maximal ideal \( \mathfrak{m}\) and another when \(\mathfrak{p}=\mathfrak{m}.\) 
\end{proof}
\subsection{One sided tilting on \(\Lambda\)}
We now revert to setting \ref{01 of 23 april 2023}, where \(\Lambda\) is derived equivalent to \(U\) where  \(f \colon U \to  \Spec R \) is a flopping contraction and \(R\) is an isolated cDV singularity. This subsection proves that \(C_{\Lambda} \in \Db(\smallmod \Lambda)\) constructed in  Proposition \ref{01 of  06 December 2022} is a one-sided tilting complex.

Recall that the  setup \ref{bi setup} defines \(\Bnohat_{i}\) and \(\Bnohat\).

\begin{lemma}\label{01 of 27 march 2023}
	If \(y\in \Db(\smallmod \Bnohat)\) is perfect, then \(y\otimes^{\bf{L}}_{\Bphat}{\Bnohat_{i}}_{\Bphat}\) is perfect.
\end{lemma}
\begin{proof}
	The exact sequence \(0 \to I_{i}\to \Bnohat\to  \Bnohat_{i} \to 0\) gives a triangle, \( y\otimes^{\bf{L}}_{\Bphat} I_{i} \to y \to y\otimes^{\bf{L}}_{\Bphat} \Bnohat_{i}\). Now,  \(-\otimes^{\bf{L}}_{\Bphat} I_{i}\) is an equivalence by \ref{infinite Ii},  so \(y\otimes^{\bf{L}}_{\Bphat} I_{i}\) is perfect. By the  2 out of 3 property \(y\otimes^{\bf{L}}_{\Bphat}{\Bnohat_{i}}_{\Bphat}\) is perfect.
\end{proof}
\begin{definition}
	Let  \(\scrT\)   be a triangulated category. Then  we say an object \(A\in \scrT\) is \emph{homologically finite} if for any object \(B \in \scrT,\) all \(\Hom_{\scrT}(A,B[i])\) are trivial  except for a finite number of \(i\in \mathbb{Z}.\)
\end{definition}
By \cite[Proposition 2.18]{IW4}, for module-finite algebras, being  a perfect complex is equivalent to being  a homologically finite complex.
\begin{prop}\label{02 0f 23 april 2023}
	\(C_{\Lambda} \in \Kb(\smallproj \Lambda). \)
\end{prop}
\begin{proof}
	The exact sequence \(0\to I_{i}\to \Bnohat\to \Bnohat_{i}\to 0  \) induces a triangle  {\scriptsize \[
		\begin{array}{c}
			\begin{tikzpicture}
				\node (a1) at (0,0) {$\left(\hat{\Lambda}\otimes \mathbb{M} \otimes \uptau_{\upalpha}\otimes I_{i}\right)_{\Bnohat}$};
				\node (a2) at (3.5,0) {$\left(\hat{\Lambda}\otimes \mathbb{M} \otimes \uptau_{\upalpha}\right)_{\Bnohat}$};
				\node (a3) at (7,0) {$\left(\hat{\Lambda}\otimes \mathbb{M}\otimes  \uptau_{\upalpha}\otimes \Bnohat_{i}\right)_{\Bnohat}$};
				\node (a4) at (9.25,0){} ;
				\node (b1) at (0,-0.8) {$\Upphi_{i}^{2}\Upphi_{\upalpha }^{-1}\mathbb{F}^{-1}(\hat{\Lambda})$};
				\node (b2) at (4,-0.8) {$\Upphi_{\upalpha }^{-1}\mathbb{F}^{-1}(\hat{\Lambda})$};
				\draw[->] (a1) to node[above] {$\scriptstyle$} (a2);
				\draw[->] (a2) to node[above] {$\scriptstyle$} (a3);
				\draw[->] (a3) to node[above] {$\scriptstyle$} (a4);
				\draw (a1) [transparent]edge node[rotate=270,opacity=1] {$\cong$} (b1);
				\draw (a2)[transparent] edge node[rotate=270,opacity=1] {$\cong$} (b2);
			\end{tikzpicture}
		\end{array}\]} in \(\Db(\smallmod \Bnohat).\)
	Since \(\Upphi_{i}, \Upphi_{\upalpha }, \mathbb{F}\) are equivalences, then  \(\Upphi_{i}^{2}\Upphi_{\upalpha }^{-1}\mathbb{F}^{-1}(\hat{\Lambda})\) and \(\Upphi_{\upalpha }^{-1}\mathbb{F}^{-1}(\hat{\Lambda})\) are  perfect as  complexes of right \(\Bnohat\)-modules. By  the 2 out of 3 property, \((\hat{\Lambda}\otimes \mathbb{M}\otimes  \uptau_{\upalpha}\otimes \Bnohat_{i})_{\Bnohat}\) is  thus perfect as a  complex of right \(\Bnohat\)-modules. 
	
	Now set  
	\begin{align*} 
		M&\colonequals  \hat{\Lambda} \otimes^{\bf{L}}_{\hat{\Lambda}}  \mathbb{M}\otimes^{\bf{L}}_{\Aphat}  \uptau_{\upalpha}\otimes^{\bf{L}}_{\Bphat} \Bnohat_{i} \otimes^{\bf{L}}_{\Bphat_{i}}\Bnohat_{i}\otimes^{\bf{L}}_{\Bphat}\uptau_{\upalpha}^{*}\otimes^{\bf{L}}_{\Aphat}\mathbb{M}^{*}\\
		&= \hat{\Lambda} \otimes^{\bf{L}}_{\hat{\Lambda}}  \mathbb{M}\otimes^{\bf{L}}_{\Aphat}  \uptau_{\upalpha}\otimes^{\bf{L}}_{\Bphat} Z_{\upalpha,i} \tag{by definition of \(Z_{\upalpha,i}\)},  \end{align*} 
	so  \eqref{01 of 24 january 2023} becomes 
	\(C\to {}_{\Lambda}\Lambda_{\Lambda} \to {}_{\Lambda}M \otimes^{\bf{L}}_{\hat{\Lambda}} \hat{\Lambda}_{\Lambda} \to C[1].\) Now observe that 
	\[ 
	M_{\hat{\Lambda}}\cong \mathbb{F}\Upphi_{\upalpha}\left( \Upphi_{\upalpha}^{-1}\mathbb{F}^{-1}(\hat{\Lambda}) \otimes^{\bf{L}}_{\Bphat}{\Bnohat_{i}}_{\Bphat}\right).
	\]   
	Since \(\Upphi_{\upalpha}^{-1}\mathbb{F}^{-1}(\hat{\Lambda})\)  is perfect by above,  \(\Upphi_{\upalpha}^{-1}\mathbb{F}^{-1}(\hat{\Lambda}) \otimes^{\bf{L}}_{\Bphat}{\Bnohat_{i}}_{\Bphat}\) is perfect  by Lemma \ref{01 of 27 march 2023}. Since \(\Upphi_{\upalpha }, \mathbb{F}\) are equivalences,  \(M\in \Kb(\smallproj \hat{\Lambda}).\)

	To show that \(M\otimes \hat{\Lambda}_{\Lambda} \in \Kb(\smallproj \Lambda)\), note  that the complex \(M\in \Dbfl (\smallmod \hat{\Lambda})\)   since \(\Bnohat_{i}\in \fl \Bnohat.\) Thus 
	\begin{align*}
		0&= \Hom_{\Db(\Lambda)}(M \otimes \hat{\Lambda}_{\Lambda}, x[t])\\
		&\stackrel{\ref{03 0f 06 march 2023}}{\Longleftrightarrow} \Hom_{\Db(\hat{\Lambda})}(M \otimes \hat{\Lambda} \otimes_{\Lambda} \hat{\Lambda}, \hat{x}_{\mathfrak{m}}[t])=0\tag{since $\Supp M\otimes \hat{\Lambda}=\{ \mathfrak{m}\}$}\\
		&\stackrel{\ref{04 0f 24 january 2023}\eqref{04 0f 24 january 2023 two}}{\Longleftrightarrow}  \Hom_{\Db(\hat{\Lambda})}(M, \hat{x}_{\mathfrak{m}}[t])=0
	\end{align*}

	Since \(M\in \Kb(\smallproj \hat{\Lambda})\) by above, it is homologically finite. Thus, by above, \(M\otimes \hat{\Lambda}\) is also  homologically finite, which in turn implies it is perfect. Considering \eqref{01 of 24 january 2023}, by the 2 out of 3 property, \(C_{\Lambda} \in \Kb(\smallproj \Lambda),\) since  both \(\Lambda\) and  \(M\otimes \hat{\Lambda}\)  are.
\end{proof}
\begin{lemma}\label{01 of 20 march 2023}
	As a right \(\Lambda\)-module, \(\Ext^{t}_{\Lambda}(C_{\Lambda}, C_{\Lambda})=0\) for all \(t\neq 0.\) Thus C  is one sided tilting complex.
\end{lemma}
\begin{proof}
	As a triangle of right \(\Lambda\)-modules, \eqref{01 of 24 january 2023} can be rewritten as 
	\begin{equation*}
		C_{\Lambdaphat} \to \Lambda_{\Lambdaphat} \to \hat{\Lambda} \otimes^{\bf{L}}_{\hat{\Lambda}} \left(  \mathbb{M}\otimes^{\bf{L}}_{\Aphat}  \uptau_{\upalpha}\otimes^{\bf{L}}_{\Bphat} Z_{\upalpha,i} \right)\otimes^{\bf{L}}_{\hat{\Lambda}} \hat{\Lambda}_{\Lambdaphat}\to, 
	\end{equation*} 
	which when localized  at a prime ideal \(\mathfrak{p}\) becomes 
	\begin{equation}\label{02 of 20 march 2023}
		C_{\Lambdaphat}\otimes_{R} R_{\mathfrak{p}} \to \Lambda_{\Lambdaphat}\otimes_{R} R_{\mathfrak{p}}  \to \left( \hat{\Lambda} \otimes^{\bf{L}}_{\hat{\Lambda}}  \mathbb{M}\otimes^{\bf{L}}_{\Aphat}  \uptau_{\upalpha}\otimes^{\bf{L}}_{\Bphat} Z_{\upalpha,i} \otimes^{\bf{L}}_{\hat{\Lambda}} \hat{\Lambda}_{\Lambdaphat} \right)\otimes_{R} R_{\mathfrak{p}} \to 
	\end{equation} 
	
	When \(\mathfrak{p} \neq \mathfrak{m},\) since \(\hat{\Lambda}\otimes\mathbb{M}\otimes \uptau_{\upalpha} \otimes Z_{\upalpha,i} \otimes \hat{\Lambda}\) is supported at only the maximal ideal \(\mathfrak{m}\),  \((\hat{\Lambda}\otimes\mathbb{M}\otimes \uptau_{\upalpha} \otimes Z_{\upalpha,i} \otimes \hat{\Lambda})_{\mathfrak{p}}=0.\) Thus  \(C_{\mathfrak{p}}\cong \Lambda_{\mathfrak{p}}\), which is tilting in \(\Db(\Lambda_{\mathfrak{p}}).\)
	
	When  \(\mathfrak{p} = \mathfrak{m},\)  by the last line in the  proof of Proposition \ref{02 0f 06 march 2023}\eqref{02 0f 06 march 2023 two},  there is an isomorphism of   \(\Lambda\)-\(\hat{\Lambda}\) bimodules
	\[ 
	C \otimes_{\Lambda} \hat{\Lambda} \cong \hat{\Lambda}\otimes \mathbb{M} \otimes \uptau_{\upalpha} \otimes I_{i} \otimes \uptau_{\upalpha}^{*}\otimes \mathbb{M}^{*}.
	\] Considering these   as simply  right  \(\hat{\Lambda}\)-modules, it follows that  \[C\otimes_{R} R_{\mathfrak{m}} \otimes \hat{R} \cong  \mathbb{F}\Upphi_{\upalpha}\Upphi_{i}^{2}\Upphi_{\upalpha}^{-1}\mathbb{F}^{-1} (\hat{\Lambda}).\] Since \(\hat{\Lambda}\) is tilting in \(\Db(\hat{\Lambda})\)  and \(\mathbb{F}, \Upphi_{\upalpha }, \Upphi_{i}\) are all equivalences, it follows that \(C\otimes_{R} R_{\mathfrak{m}} \otimes \hat{R}\) is also tilting. Since \(C_{\Lambda} \in \Kb(\smallproj \Lambda)\)  by Proposition \ref{02 0f 23 april 2023}, using Corollary \ref{02 of 16 march 2023}  it follows that \(C\) is a one-sided tilting complex.
\end{proof}
\subsection{Two sided tilting on $\Lambda$}
\begin{definition}
	A bimodule \({}_{\Gamma}X_{\Lambda}\)  is called a \emph{two-sided tilting complex} if 
	\begin{enumerate}
		\item \( \Gamma \to \Hom_{\D(\Lambda)}(X, X)\) induced by \(-\otimes_{\Gamma}^{\bf{L}}X\)  is an isomorphism.
		\item \(X_{\Lambda}\)  is a one sided tilting complex in the sense of Definition \textnormal{\ref{ 01 0f 21 march 2023}}.
	\end{enumerate}
\end{definition}

\begin{theorem}\label{twosided tilt thm}
	\({}_{\Lambda}C_{\Lambda}\) is a two-sided tilting complex.
\end{theorem}
\begin{proof}
	We know \(C_{\Lambda}\) is one sided tilting by Lemma \ref{01 of 20 march 2023}. Thus  it suffices to prove that 
	\[f\colon \Lambda \xrightarrow{-\otimes_{\Lambda}^{\bf{L}}C} \Hom_{\D(\Lambda)}(C, C)\] is an isomorphism. Isomorphisms of \(R\)-modules can be checked locally. So,  it suffices to show that \begin{equation}\label{01 0f 07 march 2023}\Lambda_{\mathfrak{p} } \xrightarrow{f_{\mathfrak{p}}=-\otimes^{\bf{L}}_{\Lambda_{\mathfrak{p}}}C_{\mathfrak{p} }} \Hom_{\D(\Lambda_{\mathfrak{p} })}(C_{\mathfrak{p} }, C_{\mathfrak{p} }),\end{equation}  is an isomorphism, for all \(\mathfrak{p} \in \Spec R.\)
	
	When \(\mathfrak{p} \neq \mathfrak{m}\), considering  the triangle \eqref{01 of 24 january 2023}, \({}_{\Lambda_{\mathfrak{p}}}C_{\mathfrak{p}_{\Lambda_{\mathfrak{p}}}} \cong {}_{\Lambda_{\mathfrak{p}}}\Lambda_{\mathfrak{p}_{\Lambda_{\mathfrak{p}}}},\) thus  \(f_{\mathfrak{p}}\) is an isomorphism.

	When \(\mathfrak{p} = \mathfrak{m}\), now \(f_{\mathfrak{m}}\) is an  isomorphism if and only if \(\hat{f}_{\mathfrak{m}}= \hat{f}\) is an isomorphism. By the universal property of completion  \cite[p.57]{M}, \(\hat{f}\) is a continuous ring homomorphism  such that 
	\[{\scriptsize
		\begin{array}{c}
			\begin{tikzpicture}
				\node (b1) at (0,-0.5) {$ \Hom_{\D(\Lambda)}(\Lambda, \Lambda) \otimes \hat{R}$};
				\node (b6) at (4,-0.5) {$\Hom_{\D(\Lambda)}(C, C) \otimes \hat{R}$};
				\node (c1) at (0,-2) {$ \Hom_{\D(\Lambda)}(\Lambda, \Lambda)$};
				\node (c6) at (4,-2) {$\Hom_{\D(\Lambda)}(C_{\Lambda}, C_{\Lambda})$};
				\draw[->] (b1) to node[above] {$ \exists ! \hat{f}$} (b6);
				\draw[->] (c1) to node[above] {$\scriptstyle f $} (c6); 
				\draw[->] (c1) to node[above] {$ $} (b1);	
				\draw[->] (c6) to node[above] {$ $} (b6);		
			\end{tikzpicture}			
	\end{array}}
	\] commutes.  By Lemma \ref{01 of 13 march 2023}, since \(C_{\Lambda}\) is perfect
	\[{\scriptsize
		\begin{array}{c}
			\begin{tikzpicture}
				\node (empty1) at (-1.2,-0.5){};
				\node (empty2) at (-1,-3){};
				\node (empty3) at (5.2,-0.5){};
				\node (empty4) at (5.2,-3){};
				\node (a1) at (0,-0.5) {$ \Hom_{\D(\Lambda)}(\Lambda, \Lambda)\otimes \hat{R}$};
				\node (a6) at (4,-0.5) {$\Hom_{\D(\Lambda)}(C, C)\otimes \hat{R}$};
				\node (b1) at (0,-1.5) {$ \Hom_{\D(\hat{\Lambda})}(\hat{\Lambda}, \hat{\Lambda})$};
				\node (b6) at (4,-1.5) {$\Hom_{\D(\hat{\Lambda})}(\hat{C}, \hat{C})$};
				\node (c1) at (0,-3) {$ \Hom_{\D(\Lambda)}(\Lambda, \Lambda)$};
				\node (c6) at (4,-3) {$\Hom_{\D(\Lambda)}(C_{\Lambda}, C_{\Lambda})$};
				\draw[->] (a1) to node[above] {$\hat{f}$} (a6);
				\draw[->] (c1) to node[left] {$ \mathrm{F}^{\LA}$} (b1);
				\draw[->] (c6) to node[right] {$ \mathrm{F}^{\LA}$} (b6);
				\draw [<-](a1) to node[right] {$\cong $} node[left] {$r$} (b1);
				\draw[<-] (a6) to node[right] {$\cong $} node[left] {$s$ }(b6);
				\draw[->] (c1) to node[above] {$\scriptstyle f $} (c6);	
				\draw [->, bend left=15] (empty2) edge (empty1);
				\draw [->, bend right=20] (empty4) edge (empty3);		 
			\end{tikzpicture}			
	\end{array}}
	\] commutes. But by Corollary \ref{01 0f 02 march 2023} and setting \[ {\scriptsize \mathrm{G}_{\upalpha,i}\colonequals \D(\hat{\Lambda})\xrightarrow{\mathbb{F}^{-1}} \D(\Anohat) \xrightarrow{\Upphi_{\upalpha}^{-1}} \D(\Bnohat)  \xrightarrow{\Upphi_{i}^{2}} \D(\Bnohat) \xrightarrow{\Upphi_{\upalpha}} \D(\Anohat) \xrightarrow{\mathbb{F}} \D(\hat{\Lambda})}, \] the diagram
	\[{\scriptsize
		\begin{array}{c}
			\begin{tikzpicture}
				\node (b1) at (0,-0.5) {$\Hom_{\D(\hat{\Lambda})}(\hat{\Lambda},\hat{\Lambda}) $};
				\node (b6) at (4,-0.5) {$ \Hom_{\D(\hat{\Lambda})}(\hat{C},\hat{C})$};
				\node (c1) at (0,-2) {$\Hom_{\D(\Lambda)}(\Lambda,\Lambda) $};
				\node (c6) at (4,-2) {$\Hom_{\D(\Lambda)}(C,C)  $};
				\draw[->] (b1) to node[above] {$ \scriptstyle \mathrm{G}^{-1}_{\upalpha,i}$} (b6);
				\draw[->] (c1) to node[above] {$ f$} node[below] {$-\otimes C$}(c6); 
				\draw[->] (c1) to node[left] {$ \mathrm{F}^{\LA}$} (b1);	
				\draw[->] (c6) to node[right] {$ \mathrm{F}^{\LA} $} (b6);					
			\end{tikzpicture}			
	\end{array}}
	\] commutes, and thus  by uniqueness, \(\hat{f}=s\circ \mathrm{G}^{-1}_{\upalpha,i}\circ r^{-1}.\) Since \(r,s\) are bijections, and  so is \(\mathrm{G}^{-1}_{\upalpha,i}\), it follows that the composition \(\hat{f}\) is an isomorphism  as required.
\end{proof}

The following is part of the main result of  \cite{R2}, specifically \cite[Theorem 1.1]{R3}.
\begin{theorem}\label{rickard thm}
	Let \(\Gamma\) and $\Lambda$ be module finite \(R\)-algebras,  and \(X\)   a complex  of  \(\Gamma\)-\(\Lambda\) bimodules. The following are equivalent.
	\begin{enumerate}
		\item   \(-\otimes^{\bf{L}}_{\Gamma} X \colon \D(\Mod \Gamma) \to \D(\Mod \Lambda)\) is an equivalence.
		\item \(-\otimes^{\bf{L}}_{\Gamma} X\) induces an equivalence  \(\Kb(\smallproj\Gamma) \to \Kb(\smallproj \Lambda).\)
		\item \(-\otimes^{\bf{L}}_{\Gamma} X\) induces an equivalence   \(\Db(\smallmod \Gamma ) \to \Db(\smallmod \Lambda).\) 
		\item  \(X\)  is a two sided tilting complex.
	\end{enumerate}
\end{theorem}

\begin{cor}\label{Twist is an equiv}
	For any  choice  \((\upalpha, i)\)   as in Setup \textnormal{\ref{bi setup}}, \({}_{\Lambda}C_{\Lambda}\)  in  \(\S\ref{twistfun}\)  is a two-sided tilting complex,  giving  rise to the autoequivalence \(\TW_{\upalpha,i}\)  of \(\Db (\smallmod  \Lambda).\) 
\end{cor}
\begin{proof}
	\({}_{\Lambda}C_{\Lambda}\) as constructed in  Proposition \ref{01 of  06 December 2022}  is a two sided tilting complex on \(\Lambda\) by Theorem \ref{twosided tilt thm}. Thus \({}_{\Lambda}C_{\Lambda}\) induces  a derived autoequivalence  on the  triangulated category \(\Db (\smallmod  \Lambda),\) by applying Theorem \ref{rickard thm}.
\end{proof}
\subsection{Corollaries}
The following are immediate.
\begin{remark}\label{inverse of twist}
	\(\TW_{\upalpha,i}^{-1}=\TW_{\upalpha,i}^{*},\) since the adjoint  to an equivalence is necessarily the  inverse.
\end{remark}
\begin{notation}
	As in the proof of Theorem \ref{twosided tilt thm}, set \[ {\scriptsize \mathrm{G}_{\upalpha,i}\colonequals \D(\hat{\Lambda})\xrightarrow{\mathbb{F}^{-1}} \D(\Anohat) \xrightarrow{\Upphi_{\upalpha}^{-1}} \D(\Bnohat)  \xrightarrow{\Upphi_{i}^{2}} \D(\Bnohat) \xrightarrow{\Upphi_{\upalpha}} \D(\Anohat) \xrightarrow{\mathbb{F}} \D(\hat{\Lambda})}. \]
\end{notation}
\begin{lemma}\label{inversecommutative}
	There are commutative diagrams 
	\[
	\begin{array}{c}
		\begin{tikzpicture}
			\node at (0,0) {$
				\begin{tikzpicture}
					\node (a1) at (0,0) {$\D(\hat{\Lambda})$};
					\node (a2) at (2,0) {$\D(\hat{\Lambda})$};
					\node (b1) at (0,-1.5) {$\D(\Lambda)$};
					\node (b2) at (2,-1.5) {$\D(\Lambda)$};
					\draw[->] (a1) to node[above] {$\scriptstyle  \mathrm{G}_{\upalpha,i}$} (a2);
					\draw[->] (b1) to node[above] {$\scriptstyle \TW_{\upalpha,i}$} (b2);
					\draw[->] (a1) to node[right] {$\scriptstyle \mathrm{F}$} (b1);
					\draw[->] (a2) to node[left] {$\scriptstyle \mathrm{F} $} (b2);						
				\end{tikzpicture}
				$}; \qquad 
			\node at (3,0) {$
				\begin{tikzpicture}
					\node (a1) at (0,0) {$\D(\hat{\Lambda})$};
					\node (a2) at (2,0) {$\D(\hat{\Lambda})$};
					\node (b1) at (0,-1.5) {$\D(\Lambda)$};
					\node (b2) at (2,-1.5) {$\D(\Lambda)$};
					\draw[->] (a1) to node[above] {$\scriptstyle  \mathrm{G}_{\upalpha,i}^{-1}$} (a2);
					\draw[->] (b1) to node[above] {$ \scriptstyle \TW_{\upalpha,i}^{-1}$} (b2);
					\draw[->] (a1) to node[right] {$\scriptstyle \mathrm{F}$} (b1);
					\draw[->] (a2) to node[left] {$\scriptstyle \mathrm{F} $} (b2);						
				\end{tikzpicture}
				$}; \qquad
			\node at (6,0) {$
				\begin{tikzpicture}
					\node (a1) at (0,0) {$\D(\hat{\Lambda})$};
					\node (a2) at (2,0) {$\D(\hat{\Lambda})$};
					\node (b1) at (0,-1.5) {$\D(\Lambda)$};
					\node (b2) at (2,-1.5) {$\D(\Lambda)$};
					\draw[->] (a1) to node[above] {$\scriptstyle  \mathrm{G}_{\upalpha,i}^{-1}$} (a2);
					\draw[->] (b1) to node[above] {$\scriptstyle\TW_{\upalpha,i}^{-1}$} (b2);
					\draw[->] (b1) to node[right] {$\scriptstyle \mathrm{F}^{\LA}$} (a1);
					\draw[->] (b2) to node[left] {$\scriptstyle \mathrm{F}^{\LA} $} (a2);					
				\end{tikzpicture}			
				$};\qquad 
			\node at (9,0) {$
				\begin{tikzpicture}
					\node (a1) at (0,0) {$\D(\hat{\Lambda})$};
					\node (a2) at (2,0) {$\D(\hat{\Lambda})$};
					\node (b1) at (0,-1.5) {$\D(\Lambda)$};
					\node (b2) at (2,-1.5) {$\D(\Lambda)$};
					\draw[->] (a1) to node[above] {$\scriptstyle  \mathrm{G}_{\upalpha,i}$} (a2);
					\draw[->] (b1) to node[above] {$\scriptstyle\TW_{\upalpha,i}$} (b2);
					\draw[->] (b1) to node[right] {$\scriptstyle \mathrm{F}^{\LA}$} (a1);
					\draw[->] (b2) to node[left] {$\scriptstyle \mathrm{F}^{\LA} $} (a2);					
				\end{tikzpicture}			
				$};
		\end{tikzpicture}
	\end{array}
	\]
\end{lemma}
\begin{proof}
	The first and third diagrams are Proposition \ref{02 0f 06 march 2023}\eqref{02 0f 06 march 2023 two} and Corollary \ref{01 0f 02 march 2023} respectively.  The second follows  from the first since \(\mathrm{G}_{\upalpha,i}\) and \(\TW_{\upalpha, i}\) are invertible. Indeed, by Proposition \ref{02 0f 06 march 2023}\eqref{02 0f 06 march 2023 two}   \begin{equation}\label{many commute}
		\TW_{\upalpha,i} \circ \mathrm{F} \circ \mathrm{G}_{\upalpha,i}^{-1}= \mathrm{F} \circ  \mathrm{G}_{\upalpha,i} \circ  \mathrm{G}_{\upalpha,i}^{-1}
		=\mathrm{F}
	\end{equation}
	thus  pre-composing \eqref{many commute} with \(\TW_{\upalpha,i}^{-1},\) gives the desired result \(\mathrm{F}\circ \mathrm{G}_{\upalpha,i}^{-1} = \TW_{\upalpha,i}^{-1} \circ \mathrm{F}.\) The fourth follows from the third similarly.
\end{proof}
\section{Group actions on $ \Db(\coh U)$ }
In this section we study the twist functor defined in $\S$\ref{twistfun} from a geometric viewpoint. We use this  to produce a group action on the  derived category of coherent sheaves of \(U,\) where the irreducible rational curves may or may not be individually floppable. 

\subsection{Geometric Twist}\label{Twist on U} Let \(f \colon U\to \Spec R \)   be an algebraic flopping contraction  as in  \ref{01 of 23 april 2023}. Now, fixing notation as in $\S$\ref{geometric setup},    there is a tilting bundle \(\mathcal{V}=\mathcal{O}_{U} \oplus \mathcal{N}\) on \(U\) inducing a derived equivalence 
\begin{eqnarray}\label{Geometric cone}
	\begin{array}{c}
		\begin{tikzpicture}[looseness=1,bend angle=15]
			\node (a1) at (0,0) {$\Db(\coh U)$};
			\node (a2) at (8,0) {$\Db(\smallmod \Lambda )$};
			\draw[->] (a1) -- node[gap] {$\scriptstyle \Upomega \colonequals\RHom_{U}(\mathcal{V},-)$} (a2);
			\draw[->,bend right] (a2) to node[above] {$\scriptstyle \Upomega^{\LA} \colonequals -\otimes^{\bf{L}}_{\Lambda} \mathcal{V} $} (a1);
			\draw[->,bend left] (a2) to node[below] {$\scriptstyle \Upomega^{\RA}\colonequals -\otimes^{\bf{L}}_{\Lambda} \mathcal{V}$} (a1);
		\end{tikzpicture}
	\end{array}
\end{eqnarray}
\begin{definition}
	\begin{enumerate}
		\item The \emph{geometric twist  functor} denoted by \(\Geo_{\upalpha,i}\) is defined to be the  composition 
		\[\begin{array}{c}
			\begin{tikzpicture}
				\node (a1) at (0,0) {$\Db(\coh U)$};
				\node (a2) at (4,0) {$\Db(\smallmod \Lambda)$};
				\node (b1) at (0,-1.5) {$\Db(\coh U)$};
				\node (b2) at (4,-1.5) {$\Db(\smallmod \Lambda)$};
				\draw[->] (a1) to node[above] {$\scriptstyle  \RHom_{U}(\mathcal{V},-)$} (a2);
				\draw[->] (b2) to node[above] {$ \scriptstyle -\otimes^{\bf{L}}_{\Lambda} \mathcal{V}$} (b1);
				\draw[->,dashed] (a1) to node[left] {$\scriptstyle \Geo_{\upalpha,i}$} (b1);
				\draw[->] (a2) to node[right] {$\scriptstyle \TW_{\upalpha,i}$} (b2);						
			\end{tikzpicture}
		\end{array}\]
		\item The \emph{inverse geometric twist functor} denoted by \(\Geo^{*}_{\upalpha,i}\) can  be defined similarly, by composing with \(\TW^{*}_{\upalpha,i}.\)
	\end{enumerate}
\end{definition}
\begin{remark}
	Being  compositions of equivalences, \(\Geo_{\upalpha,i}\)  and \(\Geo^{*}_{\upalpha,i}\) are also equivalences. Moreover \(\Geo^{*}_{\upalpha,i}\cong \Geo_{\upalpha,i}^{-1}.\)
\end{remark}

As the natural way to describe  functors between   derived categories of coherent sheaves of varieties, we now describe the functors  \(\Geo_{\upalpha,i}\) and \(\Geo^{*}_{\upalpha,i}\) as  Fourier–Mukai transforms.

Since both  \(\mathcal{V} \) and \(\mathcal{V}^{*}\) are  tilting bundles on \(U,\) adopting the commutative diagram in  \cite[6.K]{DW1} 
\[\begin{array}{c}
	\begin{tikzpicture}
		\node (a1) at (0,0) {$U\times U$};
		\node (a2) at (3,0) {$U$};
		\node (b1) at (0,-1.5) {$U$};
		\node (b2) at (3,-1.5) {$\mathrm{pt}$};
		\draw[->] (a1) to node[above] {$\scriptstyle p_{1}$} (a2);
		\draw[->] (b1) to node[below] {$ \scriptstyle \uppi_{U}$} (b2);
		\draw[->] (a1) to node[left] {$ \scriptstyle p_{2}$} (b1);
		\draw[->] (a2) to node[right] {$\scriptstyle \uppi_{U}$} (b2);	
		\draw[->] (a1) to node[above] {$\scriptstyle \uppi_{U\times U}$} (b2);						
	\end{tikzpicture}
\end{array},\] set \(\scrV \XBox \scrV^{*}=p_{1}^{*}\scrV \otimes^{\bf{L}}_{\scrO_{U\times U}} p_{2}^{*}\scrV^{*}.\) Then as in \cite{BH} we recall that 
\[\End_{U\times U}(\scrV \XBox \scrV^{*})= \Lambda\otimes_{\mathbb{C}} \Lambda^{\mathrm{op}}, \] the enveloping algebra of \(\Lambda\). Thus by basic theory there exists a derived equivalence 
\[\begin{array}{c}
	\begin{tikzcd}[ column sep=3cm, cramped, trim left]
		& \Db (\coh U\times U)  \ar[r, shift left,"{\RHom_{U\times U }(\scrV\XBox \scrV^{*},-)}",description]
		& \Db(\smallmod  \Lambda\otimes_{\mathbb{C}} \Lambda^{\mathrm{op}} ) \ar[l,shift left, "- \otimes^{\bf{L}}_{\Lambda\otimes_{\mathbb{C}}\Lambda^{\mathrm{op}}} \scrV \XBox \scrV^{* }"]
	\end{tikzcd}  
\end{array}
\]

Now, following what is laid out on \cite[p.36 and p.37]{DW1} and recalling that \({}_{\Lambda}C_{\Lambda}\) is a bimodule complex on \(\Lambda, \) we define the following

\begin{definition} Consider  \textit{geometric twist kernel }
	\[ \scrW \colonequals \RHom_{\Lambda}(C,\Lambda) \otimes^{\bf{L}}_{\Lambda\otimes_{\mathbb{C}}\Lambda^{\mathrm{op}}} \scrV \XBox \scrV^{* },\]
	and  the \textit{inverse geometric twist kernel}
	\[ \scrW^{*} \colonequals C \otimes^{\bf{L}}_{\Lambda\otimes_{\mathbb{C}}\Lambda^{\mathrm{op}}} \scrV \XBox \scrV^{*}.\]
\end{definition}

The names are justified by following Lemma.
\begin{lemma}
	\( \Geo_{\upalpha,i} \cong \mathrm{FM}(\scrW)\) and \( \Geo_{\upalpha,i}^{*} \cong \mathrm{FM}(\scrW^{*}).\)
\end{lemma}
\begin{proof}
	The proof is word for word \cite[Lemma 6.16]{DW1}.
\end{proof}

Now, suppose that the flopping locus is a chain of curves \(C,\)   consider the open set \(W=U\setminus C\) and write \(j\colon W\to U \) for  the inclusion. The  functor  \(j^{*}\colon \Db(U) \to \Db(W)\) is  called the restriction  on \(U.\) 
\begin{lemma} \label{Geotwist and j star}
	For any \((\upalpha,i),\)  \(j^{*}\circ \Geo_{\upalpha,i} \cong j^{*}.\) 	
\end{lemma}
\begin{proof}
	Consider \(\mathbb{F}\Upphi_{\upalpha}(\Bnohat_{i}), \) which is a \(\fl  \Lambda\)-module and has finite projective dimension, and set \(\scrE_{\upalpha,i}\colonequals \mathbb{F}\Upphi_{\upalpha}(\Bnohat_{i})\otimes^{\bf{L}}_{\Lambda} \scrV \) as its image across the derived equivalence. By \cite[Proposition 7.14]{DW1} there exists an exact  triangle
	\begin{equation}\label{Geometric twist1} \RHom_{U}(\scrE_{\upalpha,i},-) \otimes^{\bf{L}}_{\Bnohat_{i}} \scrE_{\upalpha,i} \to \Id \to \Geo_{\upalpha,i} \to,  \end{equation} and applying  \(j^{*}\) to \eqref{Geometric twist1} yields \begin{equation}\label{Geometric twist2} \RHom_{U}(j^{*} \scrE_{\upalpha,i},-) \otimes^{\bf{L}}_{\Bnohat_{i}} j^{*}\scrE_{\upalpha,i} \to j^{*} \to j^{*}\Geo_{\upalpha,i} \to. \end{equation}We know that \(\scrE_{\upalpha,i}\) is supported on the curve \(C,\) indeed by Claim \ref{01 0f 20th november 2022}, \(\mathbb{F}\Upphi_{\upalpha}(\Bnohat_{i})\) is supported at only the maximal ideal \(\mathfrak{m}.\) Thus  \(j^{*}\scrE_{\upalpha,i}=0,\) so the result holds  by properties of triangles in \eqref{Geometric twist2}.
\end{proof}
\begin{cor}\label{coGeotwist and j star}
	For any \((\upalpha,i),\) \(j^{*}\circ \Geo^{-1}_{\upalpha,i} \cong j^{*}.\) 
\end{cor}
\begin{proof}
	Since \(j^{*}\circ \Geo^{-1}_{\upalpha,i} \circ \Geo_{\upalpha,i}\cong j^{*} \stackrel{\ref{Geotwist and j star}}{\cong} j^{*} \circ  \Geo_{\upalpha,i},\)  the result follows by right composing with \(\Geo^{-1}_{\upalpha,i}.\)
\end{proof}
\begin{prop}\label{geo is identity}
	Let \(H = \Geo^{\pm 1}_{\upalpha_{1},i_{1}} \hdots \Geo^{\pm 1}_{\upalpha_{t},i_{t}}  \). If \(H\) is isomorphic to \(-\otimes \scrL\) where \(\scrL\) is a line bundle, then $\scrL\cong \scrO$. In particular, \(H \cong \Id\).
\end{prop}
\begin{proof}
	First, by Lemma \ref{Geotwist and j star} and Corollary \ref{coGeotwist and j star},  \(j^{*}H(\scrO_{U})=j^{*}\scrO_{U}=\scrO_{W}.\) But \(H(\scrO_{U})\cong \scrO \otimes \scrL\) by assumption. The line bundle \(\scrL\) is reflexive and 
	\begin{align*}
		\scrL&=j_{*}j^{*}\scrL\tag{ by \cite[2.11]{S} since \(C\subseteq U\) has codimension two}\\
		&\cong j_{*}\scrO_{W}\tag{\(j^{*}\scrL \cong \scrO_{W}\)}\\
		&\cong j_{*}j^{*}\scrO_{U}\tag{\(j^{*}\scrO_{U} \cong \scrO_{W}\)}\\
		&\cong \scrO_{U}\tag{again by \cite[2.11]{S}}
	\end{align*}
	So \(\scrL \cong \scrO_{U},\) as required.
\end{proof}

Write \(\scrO_{x}\) for the skyscraper sheaf of a closed point \(x.\)
\begin{lemma}\label{sky scrapper sheaf on curve}
	For any \((\upalpha,i)\) and any \(x \notin C, \Geo^{\pm 1}_{\upalpha,i}(\scrO_{x})\cong \scrO_{x}. \) 
\end{lemma}
\begin{proof}
	Applying \eqref{Geometric twist1} to \(\scrO_{x}\) yields 
	\begin{equation}\RHom_{U}(\scrE_{\upalpha,i},\scrO_{x}) \otimes^{\bf{L}}_{\Bnohat_{i}} \scrE_{\upalpha,i} \to \scrO_{x} \to \Geo_{\upalpha,i}(\scrO_{x}) \to,  \end{equation} 
	\(\RHom_{U}(\scrE_{\upalpha,i},\scrO_{x})=0, \) since \(x \notin C\), thus by properties of triangles, \(\Geo_{\upalpha,i}(\scrO_{x})\cong \scrO_{x}.\) Applying  \(\Geo^{-1}_{\upalpha,i},\)   gives  \(\scrO_{x}\cong \Geo^{-1}_{\upalpha,i}(\scrO_{x}).\)
\end{proof}

\begin{prop}\label{action on U}
	Suppose that \(f \colon  U\to  \Spec R \) is a flopping  contraction as in Setup \ref{01 of 23 april 2023},  then there exists a commutative diagram of  group homomorphisms
	\[\begin{array}{c}
		\begin{tikzpicture}
			\node (a1) at (0,0) {$  \uppi_{1}(\mathbb{C}^{n} \setminus \scrH_{\mathbb{C}})  $};
			\node (a2) at (0,-1.5) {$  \uppi_{1}(\mathbb{C}^{n} \setminus \scrH^{\aff}_{\mathbb{C}}) $};
			\node (b1) at (3,0) {$ \mathrm{Auteq}\Db(\coh U)$};
			\draw [->] (0, -0.2) -- (0, -1.3);  
			\draw[->] (a2) to node[below] {$ \scriptstyle g^{\aff}$} (b1);		
			\draw[->] (a1) to node[above] {$ \scriptstyle g$} (b1);				
		\end{tikzpicture}
	\end{array}\] 
\end{prop}

\begin{proof}
	We prove the \(g^{\aff}\) version,  with the finite version being easier. It is known that \( \uppi_{1}(\mathbb{C}^{n} \setminus \scrH^{\aff})\) is generated by \(\ell_{\upalpha,i},\)  where these correspond to a monodromy around a hyperplane \(i\) through an atom \(\upalpha\) in the hyperplane arrangement  \(\scrH^{\aff}.\)
	
	We claim the map from the  free group generated by the  \(\ell_{\upalpha,i}\) to \(\mathrm{Auteq}\Db(\coh U)\) sending  \[\ell_{\upalpha, i} \mapsto \Geo_{\upalpha,i}\]  is a well defined group homomorphism.  It suffices to show that for any arbitrary relation \(\ell^{\pm1}_{\upalpha_{1},i_{1}}\ell^{\pm1}_{\upalpha_{2},i_{2}} \hdots \ell^{\pm1}_{\upalpha_{t},i_{t}}=1\) in \(\uppi_{1}(\mathbb{C}^{n} \setminus \scrH^{\aff}),\) then  \(g^{\aff}(\ell^{\pm1}_{\upalpha_{1},i_{1}}\ell^{\pm1}_{\upalpha_{2},i_{2}} \hdots \ell^{\pm1}_{\upalpha_{t},i_{t}})=\Id.\)  We recall that the skyscraper sheaf \(\scrO_{x}\) is a sheaf supported at a single point, and we consider two cases, namely  \( x\in C\)  and \( x\notin C.\)
	
	When \(x\in C, \) recalling the notation in \eqref{formal fibre diagram}, consider  the diagram
	\[\begin{array}{c}
		\begin{tikzpicture}
			\node (e1) at (-3,1) {$\scrO_{x} \in $};
			\node (e2) at (-1,-3) {$\scrO_{x} \in $};
			\node (d1) at (-2,1) {$\Db(\scrU)$};
			\node (d2) at (9,1) {$\Db(\scrU)$};
			\node (a1) at (0,0) {$\Db(\hat{\Lambda})$};
			\node (a2) at (7,0) {$\Db(\hat{\Lambda})$};
			\node (b1) at (0,-1.5) {$\Db(\Lambda)$};
			\node (b2) at (7,-1.5) {$\Db(\Lambda)$};
			\node (c1) at (0,-3) {$\Db(U)$};
			\node (c2) at (7,-3) {$\Db(U)$};
			\draw[->] (a1) to node[above] {$\scriptstyle \mathrm{G}^{\pm1}_{\upalpha_{1},i_{1}}\mathrm{G}^{\pm1}_{\upalpha_{2},i_{2}}\cdots \mathrm{G}^{\pm1}_{\upalpha_{t},i_{t}}$} (a2);
			\draw[->] (b1) to node[above] {$ \scriptstyle \TW^{\pm1}_{\upalpha_{1},i_{1}}\TW^{\pm1}_{\upalpha_{2},i_{2}}\cdots \TW^{\pm1}_{\upalpha_{t},i_{t}}$} (b2);
			\draw[->] (a1) to node[left] {$ $} (b1);
			\draw[->] (a2) to node[right] {$ $} (b2);		
			\draw[->] (c1) to node[above] {$ \scriptstyle \Geo^{\pm1}_{\upalpha_{1},i_{1}}\Geo^{\pm1}_{\upalpha_{2},i_{2}}\cdots \Geo^{\pm1}_{\upalpha_{t},i_{t}}$} (c2);
			\draw[->] (b1) to node[right] {$ \sim$} (c1);
			\draw[->] (b2) to node[left] {$ \sim $} (c2);		
			\draw[->] (d1) to node[right] {$ $} (a1);		
			\draw[->] (d2) to node[right] {$ $} (a2);
			\draw[->] (d2) to node[right] {$ $} (c2);	
			\draw[->] (d1) to node[right] {$ $} (c1);	
			\draw [line width = 1 pt,| ->] (e1) edge (e2);				
		\end{tikzpicture}
	\end{array}\]
	The top most rectangle commutes by  Lemma \ref{inversecommutative}, and the bottom by definition. The left and right commute by \cite{VdB}.  Further since \cite{IW7} proves the existence of the affine action in the complete local setting,  \(\mathrm{G}^{\pm1}_{\upalpha_{1},i_{1}}\mathrm{G}^{\pm1}_{\upalpha_{2},i_{2}}\cdots \mathrm{G}^{\pm1}_{\upalpha_{t},i_{t}}\cong \Id.\)
	
	Thus if   \(x\in C, \) then  \(x\in \scrU, \) so
	\[\Geo^{\pm1}_{\upalpha_{1},i_{1}}\Geo^{\pm1}_{\upalpha_{2},i_{2}}\cdots \Geo^{\pm1}_{\upalpha_{t},i_{t}}(\scrO_{x})=\scrO_{x} \] follows. If  \(x\notin C,\) then by repeated use of Lemma \ref{sky scrapper sheaf on curve},  \[\Geo^{\pm1}_{\upalpha_{1},i_{1}}\Geo^{\pm1}_{\upalpha_{2},i_{2}}\cdots \Geo^{\pm1}_{\upalpha_{t},i_{t}}(\scrO_{x})=\scrO_{x}.\]  Thus in all cases \(\scrO_{x} \) is fixed under \( \Geo^{\pm1}_{\upalpha_{1},i_{1}}\Geo^{\pm1}_{\upalpha_{2},i_{2}}\cdots \Geo^{\pm1}_{\upalpha_{t},i_{t}}.\) As in \cite[Proposition 7.18]{DW1}, also see  \cite{HUY,T}, it follows that \[\Geo^{\pm1}_{\upalpha_{1},i_{1}}\Geo^{\pm1}_{\upalpha_{2},i_{2}}\cdots \Geo^{\pm1}_{\upalpha_{t},i_{t}}\cong -\otimes \scrL, \] for some line bundle. By Proposition \ref{geo is identity} the result follows.
\end{proof}

\section{The Quasi-Projective Case}\label{section 8}
\subsection{ Autoequivalences on \(X\)}
In this  subsection we globalise the algebraic flopping  contraction \( f\colon U\to \Spec R\) to obtain autoequivalences on quasi-projective varieties.  We now consider the following setup.

\begin{setup}\label{global set up}
	Let \(h\colon X \to  X_{\con} \)  be a 3-fold  flopping contraction,  where X is quasi-projective and has only Gorenstein terminal singularities.
\end{setup}

To describe an autoequivalence on \(X, \) we consider the following, where the preimage of each point \(p_{k}\) is a finite chain of curves.
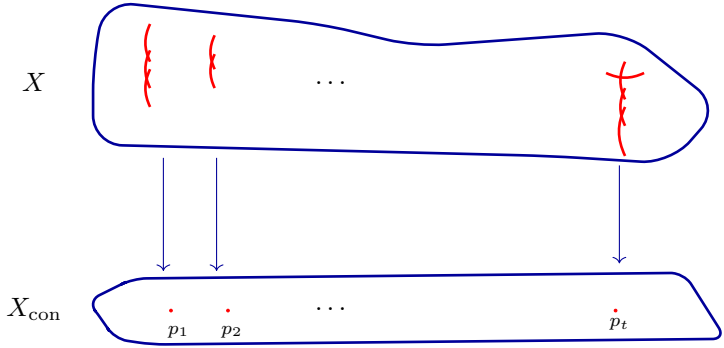
\[	\begin{tikzpicture}
	\node at (-4.9,0) {$ X$};
	\node at (-4.9,-3) {$X_{\con}$};
	\node at (-1,0) {$\hdots$};
	\node at (-1,-3) {$\hdots$};
	\node at (-3,-3.3) {$\scriptstyle p_{1}$};
	\node at (-2.3,-3.3) {$ \scriptstyle p_{2}$};
	\node at (2.8,-3.2) {$\scriptstyle p_{t}$};
	\node at (0,0) { \begin{tikzpicture}[scale=0.5]
			\coordinate (T0) at (-1.5,3);
			\coordinate (B0) at (-1.5,2);
			\coordinate (T) at (-1.5,2.3);
			\coordinate (B) at (-1.5,1.3);
			\coordinate (T1) at (-1.5,1.8);
			\coordinate (B1) at (-1.5,0.8);
			\coordinate (C0) at (0.2,2.7);
			\coordinate (N0) at (0.2,1.8);
			\coordinate (C1) at (0.2,2.2);
			\coordinate (N1) at (0.2,1.3);
			\coordinate (M0) at (11,2);
			\coordinate (K0) at (11,1);
			\coordinate (M) at (11,1.3);
			\coordinate (K) at (11,0.3);
			\coordinate (M1) at (11,0.8);
			\coordinate (K1) at (11,-0.5);
			\coordinate (M2) at (10.5, 1.7);
			\coordinate (K2) at (11.5, 1.7);
			\draw[red,line width=1pt] (B) to [bend left=25] (T);
			\draw[red,line width=1pt] (B1) to [bend left=25] (T1);
			\draw[red,line width=1pt] (B0) to [bend left=25] (T0);
			\draw[red,line width=1pt] (N1) to [bend left=25] (C1);
			\draw[red,line width=1pt] (N0) to [bend left=25] (C0);
			\draw[red,line width=1pt] (K) to [bend left=25] (M);
			\draw[red,line width=1pt] (K1) to [bend left=25] (M1);
			\draw[red,line width=1pt] (K0) to [bend left=25] (M0);
			\draw[red,line width=1pt] (K2) to [bend left=25] (M2);
			\draw[color=blue!60!black,rounded corners=11pt,line width=1pt] (-3,-0.2) -- (9,-0.5)-- (12.2,-0.7) -- (13.5,0.8)-- (11,2.8)-- (5.5,2.4) -- (3,3) -- (-2.7,3.6) -- (-3,2)-- cycle;				
	\end{tikzpicture}};
	\node at (0,-3) {\begin{tikzpicture}[scale=0.5]
			\filldraw [red] (-1.5,0.75) circle (1pt);
			\filldraw [red] (0.0,0.75) circle (1pt);
			\filldraw [red] (10.2,0.75) circle (1pt);
			\draw[color=blue!60!black,rounded corners=5pt,line width=1pt] (-3,0) -- (-2,-0.15)-- (13.1,0) -- (12,1.75)-- (-2.5,1.6) -- (-3,1.35) -- (-3.7,1) -- (-3.3,0.5)-- cycle;
	\end{tikzpicture}};
	\draw[->, color=blue!60!black] (-3.2,-1) -- (-3.2,-2.5);
	\draw[->, color=blue!60!black] (-2.5,-1) -- (-2.5,-2.5);
	\draw[->, color=blue!60!black] (2.8,-1.1) -- (2.8,-2.4);
\end{tikzpicture}
\]

Associated to each singular point \(p_{k}\)  is a corresponding finite hyperplane arrangement \(\scrH_{k}\)  and  a corresponding infinite hyperplane arrangement \(\scrH_{k}^{\aff}.\) Most of what follows is adapted from the techniques in  \cite{DW4}, and  to make this paper self-contained we summarise  it below.  Choose an affine open subset \(\Spec R\) of \(X_{\con}\) containing  only one \(p_{k},\) and set \(U\) to be its preimage.
\begin{equation}\label{key open diagram}
	\begin{array}{c}	\begin{tikzpicture}
			\node at (-4.5,0) {$ X$};
			\node at (-4.3,-3) {$X_{\con}$};
			\node at (-3.2,0.2) {$ \scriptstyle U $};
			\node at (-1.9,-2.9) {$ \scriptstyle \Spec R $};
			\node at (-2.5,0) { \begin{tikzpicture}[scale=0.5]
					\coordinate (C0) at (2.2,2.7);
					\coordinate (N0) at (2.2,1.8);
					\coordinate (C1) at (2.2,2.2);
					\coordinate (N1) at (2.2,1.3);
					\draw[red,line width=1pt] (N1) to [bend left=25] (C1);
					\draw[red,line width=1pt] (N0) to [bend left=25] (C0);
					\draw[color=blue!60!black,rounded corners=5pt,line width=1pt] (0,0.2) -- (2,0.5)-- (5.1,0.2) -- (5.8,2.1)-- (5.5,3.8) -- (3,3.3) -- (0.7,3.6) -- (-0.7,2.6)-- cycle;					
			\end{tikzpicture}};
			\node at (-2.5,-3) {\begin{tikzpicture}[scale=0.5]
					\filldraw [red] (2.5,0.75) circle (1pt);
					\draw[color=blue!60!black,rounded corners=5pt,line width=1pt] (1,0) -- (2,0.15)-- (4.1,0) -- (4.8,0.75)-- (4.5,1.6) -- (3,1.35) -- (0.7,1.5) -- (0.3,1)-- cycle;
			\end{tikzpicture}};
			\node at (-2.5,-3) {
				\begin{tikzpicture}  
					\draw[color=black](-9,0) circle(0.2cm);
			\end{tikzpicture}};
			\node at (-2.6,0) {
				\begin{tikzpicture}  
					\draw[color=black](-9,0) circle(0.5cm);
			\end{tikzpicture}};
			\draw[->, color=blue!60!black] (-2.5,-1) -- (-2.5,-2.5);
		\end{tikzpicture}
	\end{array}
\end{equation} From this, consider the following diagram 
\[\begin{array}{c}
	\begin{tikzpicture}
		\node (a1) at (0,0) {$U $};
		\node (a2) at (3,0) {$X$};
		\node (b1) at (0,-1.5) {$\Spec R $};
		\node (b2) at (3,-1.5) {$X_{\mathrm{con}}$};
		\draw[right hook->] (b1) to node[above] {$\scriptstyle i$} (b2);
		\draw[right hook->] (a1) to node[below] {$ $} (a2);
		\draw[->] (a1) to node[left] {$ \scriptstyle f $} (b1);
		\draw[->] (a2) to node[right] {$ \scriptstyle h$} (b2);					
	\end{tikzpicture}
\end{array}\]

Below, it will be convenient to change the structure sheaf and consider  general ringed spaces \((X, \scrA). \) By basic  theory, a coherent sheaf on \(X\) is defined  with reference to the sheaf of rings that contains its geometric information.  When \(\scrA\)  is fixed, we define an abelian category of coherent sheaves  \(\coh (X, \scrA), \)  where  if \(\scrA=\scrO_{X}\) then \(\coh (X, \scrA)\cong \coh X.\) 

Now, reinterpreting \eqref{Geometric cone} implies the equivalence 
\[\Db(\coh U)\xrightarrow{\sim} \Db(\coh (\Spec R, \End_{R}(f_{*}\mathcal{V}))\,), \]where \(\End_{U}(\mathcal{V}) \cong  \End_{R}(f_{*}\mathcal{V})\cong  \Lambda\)  as in  $\S$\ref{geometric setup}. In order to lift the derived equivalence from \(\hat{\Lambda}\) to \(\Lambda,\) the main trick in \S \ref{preliminary} is to  consider  the  restriction and extension of scalars between module  categories 
\begin{equation}\label{motivation to quasi coherent}\begin{array}{c}
		\begin{tikzpicture}[looseness=1,bend angle=20]
			\node (a1) at (0,0) {$\Mod \hat{\Lambda}$};
			\node (a2) at (3,0) {$\Mod \Lambda$};
			\draw[->] (a1) -- node {$ $} (a2);
			\draw[->,bend right] (a2) to node[above] {$ $} (a1);								
		\end{tikzpicture}
\end{array} \end{equation} then replace the bimodule morphisms defined in Proposition \ref{01 of  06 December 2022} as follows
\[ \left(
\hat{\Lambda} \to {}_{\Lambda}\hat{\Lambda} \otimes \left(  \mathbb{M}\otimes  \uptau_{\upalpha}\otimes Z_{\upalpha,i}\right)\otimes\hat{\Lambda}_{\Lambda}\right) \rightsquigarrow
\left( \Lambda \to \hat{\Lambda} \to  {}_{\Lambda}\hat{\Lambda} \otimes \left(  \mathbb{M}\otimes  \uptau_{\upalpha}\otimes Z_{\upalpha,i}\right)\otimes\hat{\Lambda}_{\Lambda} \right).
\]

We now repeat this trick using sheaves of algebras. Reverting to Setup \ref{global set up}, let \(\scrP=\scrO_X \oplus \scrP_0\) be the local progenerator on \(X\) of \cite[3.31]{VdB}. As in \cite[Proposition 2.5(2)]{DW4} the fibre of \(h\) has dimension of at most one, so by \cite[Assumption 2.3]{DW4}, there is an equivalence  
\[\textbf{R} h_{*} \textbf{R}\mathcal{H}om_{X}(\scrP,-) \colon \Db(\coh X) \to \Db (\coh(X_{\mathrm{con }}, \mathcal{E}nd_{X_{\mathrm{con }}}(h_{*}\scrP))).\]  Set \(\scrA \colonequals \mathcal{E}nd_{X_{\mathrm{con }}}(h_{*}\scrP^{*}), \) so that \(\scrA \cong (\mathcal{E}nd_{X_{\mathrm{con }}}(h_{*}\scrP))^{\mathrm{op}}.  \)  There are now functors   \[\begin{array}{c}
	\begin{tikzpicture}[looseness=1,bend angle=30]
		\node (a1) at (0,0) {$ \Qcoh (\Spec R,  \End(f_{*}\scrP^{*}))$};
		\node (a2) at (5,0) {$\Qcoh (X_{\con}, \scrA)$};
		\draw[->] (a1) -- node[below] {$ i_{*}$} (a2);
		\draw[->,bend right,yshift=.2cm] (3.7,0) to node[above] {$ i^{-1}$} (2.1,0);								
	\end{tikzpicture}
\end{array} \] by 	 e.g  \cite[18.3.2]{KS}, where the inverse image functor \(i^{-1}\) is left adjoint to the push forward  \(i_{*}.\)  Similarly, there is an adjunction \[\begin{array}{c}
	\begin{tikzpicture}[looseness=1,bend angle=30]
		\node (a1) at (0,0) {$ \Qcoh \left(\Spec R,  i^{-1}\scrA \otimes_{\mathbb{C}} i^{-1}\scrA^{\mathrm{op}} \right)$};
		\node (a2) at (5.5,0) {$\Qcoh (X_{\con}, \scrA \otimes_{\mathbb{C}} \scrA^{\mathrm{op}})$};
		\draw[->] (a1) -- node[below] {$ i_{*}$} (a2);
		\draw[->,bend right,yshift=.2cm] (3.7,0) to node[above] {$ i^{-1}$} (2.2,0);								
	\end{tikzpicture}
\end{array}, \] where \( i^{-1}\scrA= \End( f_{*}\scrP^{*})\) and \( i^{-1}\scrA^{\mathrm{op}} = \End(f_{*}\scrP^{*})^{\mathrm{op}}.\)   Now  \(i^{-1}({}_{\scrA}\scrA_{\scrA})\colonequals {}_{\Lambda}\Lambda_{\Lambda},\) and for any choice of \((\upalpha,i)\) in Setup \ref{bi setup}, there is a bimodule map \(\Lambda \to Q\) by $\S$\ref{twistfun} where  \(Q\colonequals {}_{\Lambda}\hat{\Lambda} \otimes \left(  \mathbb{M}\otimes  \uptau_{\upalpha}\otimes Z_{\upalpha,i}\right)\otimes\hat{\Lambda}_{\Lambda}. \)  Re-interpreting this as a bimodule map \( i^{-1}\scrA \to Q, \)  we push this forward to give \(i_{*}i^{-1}\scrA \to i_{*}Q. \) We now play the same trick as in \eqref{motivation to quasi coherent} above,  namely we replace 
\[ (i_{*}i^{-1}\scrA \to i_{*}Q ) \rightsquigarrow \left( \scrA \to i_{*}i^{-1}\scrA \to i_{*}Q \right). \] Thus there is a bimodule map 
\begin{equation}\label{map to take cone!}
	\scrA \to {}_{\scrA}(i_{*}Q)_{\scrA}.
\end{equation}
Taking the cone in the derived category of \(\scrA\) bimodules gives a triangle 
\begin{equation}\label{global equiv triangle}
	\scrC \to \scrA \to {}_{\scrA}(i_{*}Q)_{\scrA}\to  \scrC[1].
\end{equation} 		

\begin{definition}\label{twist on X defn}
	Under the  global  Setup  \textnormal{\ref{global set up}},  for any singular point \(p_{k}\) consider the associated \(\scrH_{k}\) and \(\scrH^{\aff}_{k}.\) For any \((\upalpha,j)\)  consider \(\scrC \colonequals \scrC_{\upalpha,j}\) constructed above, and define 
	\[\TW_{X,\upalpha,j}, \TW^{*}_{X,\upalpha,j}\colon \D( \Qcoh X)\to \D(\Qcoh X),\]  by
	\(\TW_{X,\upalpha,j}=  \mathbf{R}\mathcal{H}om_{\scrA}(\scrC,-)\) and  \(\TW^{*}_{X,\upalpha,j}= -\otimes^{\bf{L}}_{\scrA} \scrC. \)
\end{definition}

\begin{theorem}\label{global equivalence}
	Under the  global  Setup  \textnormal{\ref{global set up}},  for any singular point \(p_{k}\) consider the associated \(\scrH_{k}\) and \(\scrH^{\aff}_{k}.\) Then  for any \((\upalpha,j),\)  \(\TW_{X,\upalpha,j}\) and \(\TW^{*}_{X,\upalpha,j}\) are equivalences.
\end{theorem}
\begin{proof}
	Applying  \(i^{-1}\)  to \eqref{global equiv triangle} yields the following triangle 
	\[i^{-1}\scrC \to i^{-1}\scrA \to  i^{-1}\left({}_{\scrA}Q_{\scrA}\right)\to. \]
	Under the Zariski local   Setup   \ref{01 of 23 april 2023}, we have that \(i^{-1}\scrA ={}_{\Lambda}\Lambda_{\Lambda}, 	i^{-1}\left({}_{\scrA}Q_{\scrA}\right) ={}_{\Lambda}Q_{\Lambda}\) 
	and thus consequently  \( i^{-1}\scrC= C\) from  the Zariski local setup.

	Choose an affine open cover of  \(X_{\con}\) containing \(\Spec R\) in \eqref{key open diagram}, where no other open set contains \(p_{k}\).  To ease notation,  write \(V=\Spec R,\) and consider  \[ \mathbf{R}\mathcal{H}om_{\scrA|_{V}}(\scrC|_{V},-) \colon  \D(\Mod \scrA|_{V}) \to  \D(\Mod \scrA|_{V}).\] Since \(V\) is affine, by \cite[Setup 4.1]{DW4}, \(\scrA|_{V}\) corresponds to \(\Lambda\) and \(\scrC|_{V}\)  to \({}_{\Lambda}C_{\Lambda}.\)
	
	Hence, the functor  \(  \mathbf{R}\mathcal{H}om_{\scrA|_{V}}(\scrC|_{V},-)\) becomes \[\RHom_{\Lambda}(C,-)\colon \D(\Mod \Lambda) \to \D(\Mod \Lambda) \] which is simply the equivalence \(\TW_{\upalpha,j}\) on \(\Lambda\). On the other opens \(W\) \[\mathbf{R}\mathcal{H}om_{\scrA|_{W}}(\scrC|_{W},-)=\mathbf{R}\mathcal{H}om_{\scrA|_{W}}(\scrA|_{W},-)\] since \(Q|_{W}=0.\) Thus on all opens in the  covering of \(X_{\con}, \mathbf{R}\mathcal{H}om_{\scrA}(\scrC,-)\)  restricts to an equivalence. As in \cite[$\S$5.2]{DW4}, this implies that \(\mathbf{R}\mathcal{H}om_{\scrA}(\scrC,-)\) is an equivalence. Its adjoint \(\TW^{*}_{X,\upalpha,j}\) must be the inverse, and so is also an equivalence.
\end{proof}

\subsection{Group actions on $ \Db(\coh X)$ }
The following  result which is valid at every singular point, will lead to our main result in Corollary  \ref{global group action}.
\begin{theorem}\label{ theorem fixed pk}
	For each singular point \(p_{k},\) there exists a commutative diagram of group homomorphisms
	\[\begin{array}{c}
		\begin{tikzpicture}
			\node (a1) at (0,0) {$  \uppi_{1}(\mathbb{C}^{n_{k}} \setminus (\scrH_{k})_{\mathbb{C}})  $};
			\node (a2) at (0,-1.5) {$  \uppi_{1}(\mathbb{C}^{n_{k}} \setminus (\scrH_{k}^{\aff})_{\mathbb{C}}) $};
			\node (b1) at (4,0) {$ \mathrm{Auteq}\Db(\coh X)$};
			\draw[->] (a1) to node[below] {$ $} (a2);		
			\draw[->] (a2) to node[below] {$\scriptstyle m_{k}^{\aff} $} (b1);		
			\draw[->] (a1) to node[above] {$ \scriptstyle m_{k}$} (b1);				
		\end{tikzpicture}
	\end{array}\] 
\end{theorem}
\begin{proof}
	We prove the infinite version,  with the finite version being similar. By   Proposition  \ref{action on U}, there exists a group homomorphism 
	\begin{equation} \label{group hom on uk} \uppi_{1}(\mathbb{C}^{\mathit{n}_{k}} \setminus (\scrH^{\aff}_{k})_{\mathbb{C}})  \to \mathrm{Auteq}\Db(\coh U_{k}).
	\end{equation}
	
	For any  choice \((\upalpha, j) \) in Setup \ref{bi setup}  associated to \(\scrH^{\aff}_{k},\)  temporarily write \[ \mathrm{Geo}_{k}\mathrm{Twist}_{\upalpha,j}~  \textnormal{or} ~\mathrm{Geo}_{k}\mathrm{Twist}^{-1}_{\upalpha,j}\] for the corresponding geometric  or geometric  inverse twist on \(U_{k}.\) By Theorem \ref{global equivalence}, \[\TW_{X}|_{U_{k}}\cong \mathrm{Geo}_{k}\mathrm{Twist}_{\upalpha,j} ~\textnormal{and}~ \TW^{-1}_{X}|_{U_{k}}\cong \mathrm{Geo}_{k}\mathrm{Twist}^{-1}_{\upalpha,j}.\]
	
	We next define 
	\[ m^{\aff}_{k} \colon   \uppi_{1}(\mathbb{C}^{\mathit{n}_{k}} \setminus (\scrH^{\aff}_{k})_{\mathbb{C}})   \to \mathrm{Auteq}\Db(\coh X), \] by mapping the generators \(\ell_{\upalpha,j}\)  to \(\TW_{X,\upalpha, j}\) in Definition \ref{twist on X defn}. We now prove this is a homomorphism. Suppose that  \(\ell^{\pm1}_{\upalpha_{1}, j_{1}} \ell^{\pm1}_{\upalpha_{2}, j_{2}} \cdots \ell^{\pm1}_{\upalpha_{t}, j_{t}} \) is a relation in \( \uppi_{1}(\mathbb{C}^{\mathit{n}_{k}} \setminus (\scrH^{\aff}_{k})_{\mathbb{C}}).\) We prove the corresponding relation holds in \(\mathrm{Auteq}\Db(\coh X).\)  For this, by Proposition \ref{action on U}, for a skyscraper sheaf \(\scrO_{x}\) where \(x\in U_{k},\) consider the commutative diagram
	\[\begin{array}{c}
		\begin{tikzpicture}
			\node (a1) at (-0.7,-0.8) {$\scrO_{x}  $};
			\node (a2) at (8.7, -0.8) {$\scrO_{x} $};
			\node (d1) at (-0.7,-3.4) {$ \scrO_{x}$};
			\node (d2) at (8.7,-3.4) {$ \scrO_{x}$};
			\node (b1) at (0,-1.5) {$\Db(U_{k})$};
			\node (b2) at (8,-1.5) {$\Db(U_{k})$};
			\node (c1) at (0,-3) {$\Db(X)$};
			\node (c2) at (8,-3) {$\Db(X)$};
			\node (f1) at (4,-1.35) {$ $};
			\node (f2) at (4,-1.8) {$\Id $};
			\draw[->] (c1) to node[above] {$ \scriptstyle \TW^{\pm1}_{X,\upalpha_{1},j_{1}}\TW^{\pm1}_{X,\upalpha_{2},j_{2}}\cdots \TW^{\pm1}_{X,\upalpha_{t},j_{t}}$} (c2);
			\draw[->] (b1) to node[above] {$ \scriptstyle \mathrm{Geo}_{k}\mathrm{Twist}^{\pm 1}_{\upalpha_{1},j_{1}} \mathrm{Geo}_{k}\mathrm{Twist}^{\pm 1}_{\upalpha_{2},j_{2}} \hdots \mathrm{Geo}_{k}\mathrm{Twist}^{\pm 1}_{\upalpha_{t},j_{t}} $} (b2);
			\draw[->] (b1) to node[right] {$ $} (c1);
			\draw[->] (b2) to node[left] {$ $} (c2);		
			\draw [line width = 1 pt,| ->] (a1) edge (a2);				
			\draw [line width = 1 pt,| ->] (a1) edge (d1);	
			\draw [line width = 1 pt,| ->] (a2) edge (d2);	
			\draw [line width = 1 pt,dotted, | ->] (d1) edge (d2);	
			\path (f1)--(f2) node[midway]{$\scriptstyle\parallel$};
		\end{tikzpicture}
	\end{array}\]
	where \(\scrO_{x} \mapsto \scrO_{x}\) on the top line since \eqref{group hom on uk} is a group homomorphism. Thus  \(\scrO_{x} \mapsto \scrO_{x}\) for all \(x\in U_{k}.\)  Also, \(\scrO_{x} \mapsto \scrO_{x}\) for all \(x\in X\setminus U_{k}, \) for the same reason as in  Lemma \ref{sky scrapper sheaf on curve}.  We conclude \(\scrO_{x} \mapsto \scrO_{x}\) for all \(x\in X.\)
	
	Thus, by \cite{DW1,HUY,T}, \[\TW^{\pm1}_{X,\upalpha_{1},j_{1}}\TW^{\pm1}_{X,\upalpha_{2},j_{2}}\cdots \TW^{\pm1}_{X,\upalpha_{t},j_{t}} \cong- \otimes   \scrL.\]  As in Proposition \ref{geo is identity} applied to  \(X,\) this implies \[\TW^{\pm1}_{X,\upalpha_{1},j_{1}}\TW^{\pm1}_{X,\upalpha_{2},j_{2}}\cdots \TW^{\pm1}_{X,\upalpha_{t},j_{t}} \cong \Id,\] as required.
\end{proof}

Recall that   if \( \scrH^{\aff}_{1} \)  and  \(\scrH^{\aff}_{2}\)  are   hyperplane arrangements in \(\mathbb{R}^{\mathit{a}_{1}}\) and \(\mathbb{R}^{\mathit{a}_{2}}\) respectively, then   the product   hyperplane arrangement \( \scrH^{\aff}_{1} \times\scrH^{\aff}_{2}  \) is in \(\mathbb{R}^{\mathit{a}_{1}+ \mathit{a}_{2}}.\) The following is our main result.

\begin{cor}\label{global group action}
	Suppose that \(X \to X_{\con}\) is a flopping contraction between quasi-projective 3-folds, where X has Gorenstein terminal singularities. For any singular point \(p_{k},\)   associate \(\scrH_{k}\) and  \(\scrH^{\aff}_{k} \) and set \[ \mathfrak{H}\colonequals   \scrH_{1} \times \scrH_{2} \times \hdots \times \scrH_{t}~ \text{and} ~\mathfrak{H}^{\aff}\colonequals   \scrH^{\aff}_{1} \times \scrH^{\aff}_{2} \times \hdots \times \scrH^{\aff}_{t}.\]  Then there exists a commutative diagram of  group homomorphisms
	\[\begin{array}{c}
		\begin{tikzpicture}
			\node (a1) at (0,0) {$  \uppi_{1}(\oplus \mathbb{C}^{n_{k}} \setminus \mathfrak{H}_{\mathbb{C}})  $};
			\node (a2) at (0,-1.5) {$  \uppi_{1}(\oplus \mathbb{C}^{  n_{k}} \setminus \mathfrak{H}^{\aff}_{\mathbb{C}}) $};
			\node (b1) at (4,0) {$ \mathrm{Auteq}\Db(\coh X)$};
			\draw[->] (a1) to node[below] {$ $} (a2);		
			\draw[->] (a2) to node[below] {$\scriptstyle m^{\aff} $} (b1);		
			\draw[->] (a1) to node[above] {$\scriptstyle  m$} (b1);				
		\end{tikzpicture}
	\end{array}\] 
\end{cor}
\begin{proof} 
	Recall that   \[  \uppi_{1}(\oplus \mathbb{C}^{ \mathit{n}_{k}} \setminus \mathfrak{H}_{\mathbb{C}} )=  \uppi_{1}(\mathbb{C}^{\mathit{n}_{1}} \setminus (\scrH_{1})_{\mathbb{C}}) \times \uppi_{1}(\mathbb{C}^{\mathit{n}_{2}} \setminus (\scrH_{2})_{\mathbb{C}})) \times \hdots \times \uppi_{1}(\mathbb{C}^{\mathit{n}_{t}} \setminus (\scrH_{t})_{\mathbb{C}})\] 
	and
	\[  \uppi_{1}(\oplus \mathbb{C}^{ \mathit{n}_{k}} \setminus \mathfrak{H}^{\aff}_{\mathbb{C}} )=  \uppi_{1}(\mathbb{C}^{\mathit{n}_{1}} \setminus (\scrH^{\aff}_{1})_{\mathbb{C}}) \times \uppi_{1}(\mathbb{C}^{\mathit{n}_{2}} \setminus (\scrH^{\aff}_{2})_{\mathbb{C}})) \times \hdots \times \uppi_{1}(\mathbb{C}^{\mathit{n}_{t}} \setminus (\scrH^{\aff}_{t})_{\mathbb{C}}).   \]   Set \(m=(m_{1}, m_{2}, \hdots, m_{t})\) and \(m^{\aff}=(m^{\aff}_{1}, m^{\aff}_{2},\hdots,   m^{\aff}_{t}).\)  By Theorem \ref{ theorem fixed pk}, the only thing left to prove is that the  functors from different summands of  \(\uppi_{1}\) commute, but this is   \cite[Remark 5.6]{DW3}.
\end{proof}

\end{document}